\documentclass[11pt,a4paper,twoside,reqno]{amsart}
\usepackage[T1]{fontenc}
\usepackage[utf8]{inputenc}
\usepackage{lmodern}
\usepackage{amsmath,amssymb,amsthm,mathtools}
\usepackage[margin=29mm]{geometry}
\usepackage{microtype,booktabs,array,longtable,tabularx,enumitem,xcolor,tikz}
\usetikzlibrary{calc}
\usepackage{hyperref}
\hypersetup{colorlinks=true,linkcolor=blue!50!black,citecolor=blue!50!black,urlcolor=blue!50!black,pdftitle={Sharp volume bounds in Picard number one and polynomial bounds for log Calabi--Yau surfaces},pdfauthor={Pinxian Bie}}

\setlist{leftmargin=20pt,itemsep=3pt,topsep=4pt}
\newtheorem{theorem}{Theorem}[section]
\newtheorem{proposition}[theorem]{Proposition}
\newtheorem{lemma}[theorem]{Lemma}
\newtheorem{corollary}[theorem]{Corollary}
\theoremstyle{definition}\newtheorem{example}[theorem]{Example}
\theoremstyle{remark}\newtheorem{remark}[theorem]{Remark}
\numberwithin{equation}{section}
\numberwithin{table}{section}
\newcommand{\eps}{\varepsilon}\newcommand{\Q}{\mathbb Q}\newcommand{\R}{\mathbb R}\newcommand{\Z}{\mathbb Z}\newcommand{\C}{\mathbb C}\newcommand{\PP}{\mathbb P}
\DeclareMathOperator{\vol}{vol}\DeclareMathOperator{\Sing}{Sing}\DeclareMathOperator{\mld}{mld}\DeclareMathOperator{\lcm}{lcm}\DeclareMathOperator{\ord}{ord}
\newcommand{\peff}{pseudo-effective}
\newcommand{\ellog}{\mathcal L}
\DeclareMathOperator{\Cl}{Cl}

\DeclareMathOperator{\Exc}{Exc}
\DeclareMathOperator{\Diff}{Diff}
\newcommand{\two}[1]{(2)_{#1}}
\newcommand{\ZZ}{\Z}\newcommand{\RR}{\R}\DeclareMathOperator{\Int}{Int}
\begin{document}
\setlength{\LTcapwidth}{\textwidth}
\title[Sharp and polynomial volume bounds]{Sharp volume bounds in Picard number one\protect\\ and polynomial bounds for log Calabi--Yau surfaces}
\author[Pinxian Bie]{Pinxian Bie}
\address{School of Mathematical Sciences, Fudan University, Shanghai, China}
\email{23110180002@m.fudan.edu.cn}
\date{September 26, 2026}
\begin{abstract}
For a complex $\varepsilon$-log canonical Fano surface of Picard number
one, we prove that the anticanonical volume is at least
$\varepsilon^2/100$ and that every integral nef and big Weil
divisor has volume at least $\varepsilon^7/1728$.
Both exponents are optimal. For $\varepsilon$-log canonical
log Calabi--Yau surface pairs of arbitrary Picard number, we prove a lower
bound of order $\varepsilon^9/(1+\log_2(1/\varepsilon))$,
without a bigness assumption or further restrictions on the
boundary coefficients. We also prove $\tau(X,H)\le18\varepsilon^{-3}$ for
the pseudo-effective threshold on every projective
$\varepsilon$-log canonical surface, with optimal exponent three.
A volume-to-birationality estimate gives a sufficient bound of order
$\varepsilon^{-11/2}\sqrt{1+\log_2(1/\varepsilon)}$
in the log Calabi--Yau setting.
\end{abstract}

\maketitle
\tableofcontents
\section{Introduction}\label{setup-sec}

Throughout, we work over the field of complex numbers $\C$.

The geometry of a polarised variety is intensively studied in birational geometry. It depends both on its singularities and on the volume of its polarisation. For an integral nef and big Cartier divisor on a projective surface,
the self-intersection is a positive integer. An integral Weil divisor
on a klt surface can instead have arbitrarily small rational
self-intersection. We study how a lower bound for log discrepancies
controls this failure of integrality, with particular emphasis on
Fano surfaces of Picard number one and log Calabi--Yau surface pairs.

We fix some $0<\eps\le1$. A log Calabi--Yau
$\Q$-pair is a projective pair $(X,B)$ with $B\ge0$ and
$K_X+B\sim_\Q0$. Every polarisation in this paper is an
\emph{integral Weil divisor}, assumed nef and big. Here the integrality
is essential: arbitrarily small rational rescalings would
preclude any uniform volume lower bound. Put
\[
                  \ellog(\eps)=1+\log_2(1/\eps).
\]

\subsection{Main results}
The following Table~\ref{tab:main-results} separates the sharp exponents from the
bounds whose optimal order remains open. All absolute constants
are explicit but are not optimised.

\begin{table}[htbp]
\centering\small
\caption{Quantitative bounds. Here $H$ is integral, nef and big.}
\label{tab:main-results}
\begin{tabular}{@{}>{\raggedright\arraybackslash}p{.16\textwidth}>{\raggedright\arraybackslash}p{.28\textwidth}>{\raggedright\arraybackslash}p{.32\textwidth}>{\raggedright\arraybackslash}p{.14\textwidth}@{}}
\toprule
Invariant & Setting & Bound & Exponent\\\midrule
$(-K_X)^2$ & $\eps$-lc Fano, $\rho=1$ &
$\ge\eps^2/100$ & Sharp\\
$H^2$ & $\eps$-lc Fano, $\rho=1$ &
$\ge\eps^7/1728$ & Sharp\\
$\tau(X,H)$ & Any $\eps$-lc surface &
$\le18\eps^{-3}$ & Sharp\\
$H^2$ & $\eps$-lc log Calabi--Yau pair &
$\ge c\eps^9/\ellog(\eps)$ & Open\\
Birationality & $\eps$-lc log Calabi--Yau pair &
$O(\eps^{-11/2}\sqrt{\ellog(\eps)})$ suffices & Open\\
\bottomrule
\end{tabular}
\end{table}

\begin{theorem}[Sharp quadratic anticanonical order]\label{main-rank-fano}
If $X$ is an $\eps$-lc Fano surface with $\rho(X)=1$, then
\[
                         (-K_X)^2\ge\frac{\eps^2}{100}.
\]
The exponent two is sharp.
\end{theorem}

\begin{theorem}[Sharp seventh-order polarisation bound]\label{main-rank-volume}
If $X$ is an $\eps$-lc Fano surface with $\rho(X)=1$ and $H$
is an integral nef and big Weil divisor, then
\[
                         H^2\ge\frac{\eps^7}{1728}.
\]
More generally, if $(X,B)$ is an $\eps$-lc log Calabi--Yau
$\Q$-pair and the ample model of $H$ has Picard number one,
then $H^2\ge\eps^7/55440$. The exponent seven is sharp,
already for weighted projective planes admitting such boundaries.
\end{theorem}

Theorems~\ref{main-rank-fano} and \ref{main-rank-volume}
determine the optimal exponents in Picard number one:
two for the anticanonical class and seven for arbitrary
integral Weil polarisations. The latter can have much
larger denominators, so the two volume problems have
different optimal orders.
Examples~\ref{quadratic-example} and \ref{markov-example}, together
with Theorem~\ref{seventh-example}, illustrate this distinction.

For a nef and big divisor $H$, define
\[
 \tau(X,H)=\inf\{t\ge0:K_X+tH\text{ is pseudo-effective}\}.
\]
The analogous nef threshold is denoted by $\lambda(X,H)$ and
may be infinite.

\begin{theorem}[Sharp cubic pseudo-effective threshold]\label{main-thresholds}
For every projective $\eps$-lc surface $X$ and integral nef
and big Weil divisor $H$,
\begin{equation}\label{threshold-input}
                         \tau(X,H)\le18\eps^{-3}.
\end{equation}
No assumption on $H-K_X$ or on an anticanonical boundary is
required. The exponent three is sharp, already on rank-one
Fano surfaces, where $\lambda(X,H)=\tau(X,H)$.
\end{theorem}

\begin{theorem}[Polynomial volume bound]\label{main}
There is an explicit constant $c>0$ such that, if $(X,B)$ is
an $\eps$-lc log Calabi--Yau $\Q$-surface pair and $H$ is
integral, nef and big, then
\begin{equation}\label{main-cy}
 \vol(H)=H^2\ge c\frac{\eps^9}{\ellog(\eps)}
                         \ge\frac c2\eps^{10}.
\end{equation}
The same order holds for an effective $\R$-boundary with
$K_X+B\equiv0$, after a change of constant, and for
$\eps$-lc weak Fano surfaces and $\eps$-lc log weak Fano pairs.
\end{theorem}

To the best of our knowledge, Theorem~\ref{main} gives the
first explicit polynomial lower bound in $\eps$ for the
volume of an integral nef and big Weil divisor on an
arbitrary $\eps$-lc log Calabi--Yau surface pair, without
any further restriction on the boundary coefficients or
a bigness assumption on the boundary.
The discrepancy hypothesis concerns the pair $(X,B)$.
Ambient $\eps$-lc singularities alone need
not supply an anticanonical boundary with the same discrepancy
bound. The ninth-order estimate differs from the known
seventh-order obstruction by two powers and a logarithmic
factor; optimality in arbitrary Picard number remains open.

The following conversion applies beyond the log Calabi--Yau
setting. It is proved in Theorem~\ref{volume-to-bir}.

\begin{theorem}[Quantitative volume-to-birationality conversion]
\label{intro-conversion}
Let $X$ be a projective $\eps$-lc surface and $H$ an integral
nef and big Weil divisor. If $H-K_X$ is big, there is an integer
$M$ with
\[
 M\le52\eps^{-1}\max\{1,8/\sqrt{H^2}\}
\]
such that $|mH+L|$ and $|K_X+mH+L|$ are birational for every
integer $m\ge M$ and every integral pseudo-effective Weil
divisor $L$. If $H-K_X$ is only pseudo-effective, one may
replace the maximum by $\max\{2,8/\sqrt{H^2}\}$.
\end{theorem}

\begin{corollary}[Polynomial effective birationality]\label{main-birational}
Under the $\Q$-pair hypotheses of Theorem~\ref{main}, both
$|mH+L|$ and $|K_X+mH+L|$ are birational for every integral
pseudo-effective Weil divisor $L$ and every integer
\begin{equation}\label{intro-bir-bound}
 m\ge\left\lceil C\eps^{-11/2}\sqrt{\ellog(\eps)}\right\rceil,
 \qquad C=416c^{-1/2}.
\end{equation}
In particular $m\ge\lceil\sqrt2 C\eps^{-6}\rceil$ suffices.
The same orders hold for the variants in Theorem~\ref{main},
with a change of constant.
\end{corollary}

The assertion holds for every integer above the bound, with
no divisibility condition. On a rank-one ample model,
Corollary~\ref{rank-one-bir} gives the stronger sufficient
bound $O(\eps^{-9/2})$, without logarithmic loss. We do not
claim that either birationality exponent is sharp.

\subsection{Ideas of the proofs}
The rank-one anticanonical theorem uses the classification
of Lacini, the correspondence of Palka--Pe\l ka, and
the correction discussed by Nagaoka.
Proposition~\ref{classification-reduction} lists the cases
needed over $\C$. Continued fractions,
two-section intersection formulas, and a quasi-\'etale toric
double cover control the configurations with coupled parameters.
Appendix~\ref{classification-app} gives the full parameter
ranges and the remaining numerical estimates.

The seventh-order polarisation bound uses a different invariant:
the least common multiple $q$ of the local resolution determinants
$D_p$. Since $qH$ is Cartier, $qH^2\in\Z_{>0}$. A planar
width estimate controls the toric determinant product; the
classification is used only to identify the remaining basket
with three noncanonical cyclic points. Together with the
local estimates, this gives $q\le1728\eps^{-7}$. The case
$K_X\equiv0$ has the absolute gap $1/55440$ in every Picard
number, by Proposition~\ref{zero-absolute}.

For thresholds, the local periods enter additively.
Singular Riemann--Roch expresses $\chi(\mathcal O_X(mH))$
as a quadratic polynomial plus periodic local corrections.
An annihilating polynomial of degree $3+\sum_p(D_p-1)$,
combined with vanishing at consecutive negative integers,
bounds the rank-one threshold by this degree.
Belousov's singularity bound and Lemma~\ref{r1-lem:local}
then give $18\eps^{-3}$. The MMP used for the reduction
preserves the threshold exactly. This argument does not use
Theorem~\ref{main-rank-fano} or the full classification.

For a big log Calabi--Yau boundary, write $-K_X=P+N$.
The general anticanonical estimate recalled below and the
cubic threshold give
\[
 H^2\ge\frac{P^2}{(18\eps^{-3})^2}
       \ge\frac{A_F}{324}\frac{\eps^9}{\ellog(\eps)}.
\]
If the boundary is not big, its anticanonical nef model
has either numerically trivial canonical class or a
genus-one fibration. In the rational genus-one branch,
the Kodaira configurations allow at most two modified
fibres. Their correction denominators are $O(\eps^{-1})$,
and the Halphen index cancels against the integral fibre
degree of $H$. This yields
$H\cdot P\ge\eps^2/36$ and $H^2\ge\eps^5/648$.

Finally, Lemma~\ref{centre-degree} bounds the degree of a
general curve centre in Birkar's construction. The surface
adjunction calculation gives a linear discrepancy gap,
which replaces the factorial loss in the earlier surface
alternative. The standard centre construction and cutting
lemma then yield Theorem~\ref{intro-conversion}.

\subsection{Earlier work and organisation}
Birkar's theorem on polarised varieties gives uniform
effective birationality under the adjoint condition
$H-K_X$ pseudo-effective \cite[Theorem 1.1]{Birkar}.
The dependence on $\eps$ is not explicit there. The
surface alternative in \cite[Lemma 4.1]{Bie} makes this
dependence explicit but still involves factorial growth.
We make the dependence polynomial for log Calabi--Yau
surface pairs with integral Weil polarisations.
Zhu obtains explicit estimates under the additional
condition $\dim\overline{\varphi_H(X)}\ge\dim X-1$
\cite[Theorem 1.2 and Corollary 1.7]{Zhu}; we place no
such condition on $|H|$. Jiao proves discreteness of
volumes on $\eps$-lc Calabi--Yau pairs in any fixed
dimension, independently of the boundary coefficients
\cite[Theorem 1.1]{Jiao}. That result does not give an
explicit polynomial dependence on $\eps$.

Our general anticanonical input is
\begin{equation}\label{fano-input}
 (-K_Z)^2\ge A_F\frac{\eps^3}{\ellog(\eps)}
\end{equation}
for complex $\eps$-lc weak Fano surfaces, proved in
\cite[Theorem 1.1]{Fano}. We use the local and weighted-fibre
estimates of the same paper, specifically version 2.
The rank-one quadratic theorem refines this estimate independently
of the general ninth-order proof. Even a quadratic general
anticanonical bound would give only eighth order by the
comparison in Section~1.2.

Section~\ref{prelim-sec} contains the common local inputs.
Sections~\ref{rank-one-sec} and \ref{weil-index-sec} prove
the rank-one anticanonical and polarisation bounds.
Section~\ref{local-sec} proves the threshold theorem.
Sections~\ref{big-sec}--\ref{zero-sec} treat the general
volume bound, and Section~\ref{birational-sec} proves the
birationality conversion and its applications.
Section~\ref{examples-sec} establishes sharpness of the three
optimal exponents. Appendices~\ref{classification-app}
and \ref{arithmetic-app} record the classification tables
and the finite arithmetic used in the proofs.
\par\medskip\noindent\textbf{Acknowledgements.}
The author is grateful to his adviser Meng Chen for his great encouragement and support. He thanks Minzhe Zhu for suggesting this problem and for many helpful discussions, and Zhengjie Yu for checking the draft version of this manuscript.

\par\medskip\noindent\textbf{AI disclosure.}
The author used ChatGPT as a research-assistance tool during the preparation of this manuscript. In particular, ChatGPT pro 5.6 and 6.0 assisted in exploratory discussions. They helped in the following way. First of all. they brought the theory of Halpen surfaces to the author's attention. which is one of the main tools in Section 7. They also helped the author to check through the Classification tables in Appendix A. An very useful example is found by AI in section 10. They are also used for English language editing. This paper is written by the author, who takes full responsibility for the content and accuracy of the paper.
\hfill

\section{Preliminaries}\label{prelim-sec}
\subsection{Conventions and standard results}
We use log discrepancies: a component with coefficient $b$ has discrepancy $1-b$. A pair is $\eps$-lc if every divisorial log discrepancy is at least $\eps$. In particular it is klt. We use the standard facts about klt surface singularities: they are rational quotient singularities and thus are $\Q$-factorial; their minimal resolutions consist of Hirzebruch--Jung chains or quotient forks. We use numerical pullbacks and rational intersection numbers for Weil divisors. For these facts, as well as adjunction for plt surface pairs, see \cite[Chapter 4 and Theorem 5.22]{KM}.

For an integral Weil divisor $D$, the notation $|D|$ means the complete linear system of the reflexive sheaf $\mathcal O_X(D)$. If $qD$ is Cartier, then $qD^2=(qD)\cdot D$ is an integer. The volume of a nef and big $\Q$-Cartier divisor equals its self-intersection number.

We use the cone and contraction theorems in the surface case, termination of the surface MMP, the base-point-free theorem for a klt pair, and abundance for log surfaces. Precise statements can be found in \cite[Theorems 3.3 and 3.7, Lemma 3.39]{KM} and \cite[Theorem 1.1]{Fujino}. In particular, a $K$-negative extremal ray has a rational curve $C$ with $0<-K_X\cdot C\le4$. A klt surface MMP has only divisorial contractions, until reaching a possible Mori fibre space.

\begin{table}[htbp]
\centering\small
\caption{Roles of the hypotheses.}\label{tab:hypotheses}
\begin{tabular}{@{}p{.33\textwidth}p{.61\textwidth}@{}}
\toprule
Hypothesis & Use in the argument\\\midrule
Ambient $\eps$-lc singularities &
Local determinant bounds, adjoint thresholds, and the
volume-to-birationality conversion.\\
Pair $(X,B)$ is $\eps$-lc &
Preservation of the discrepancy bound on the anticanonical
models and control of crepant fibre modifications.\\
Integrality of $H$ &
The positive integers $qH^2$, fibre degrees, and degrees
on moving curves in the smooth locus.\\
$H-K_X$ big or pseudo-effective &
Extension of the centre construction to both ordinary
and adjoint birational systems.\\
\bottomrule
\end{tabular}
\end{table}
\subsection{Quantitative bounds from the previous paper}
We retain the local determinant constant from \cite{Fano} and use the threshold in Theorem~\ref{main-thresholds}. Set
\begin{equation}\label{constants}
R(\eps)=344\eps^{-3},\qquad T(\eps)=18\eps^{-3}.
\end{equation}
The constant in \eqref{fano-input} may be taken to be
\[
c_{\mathrm{exc}}=(2\cdot84^{128\cdot42^5})^{-1},
\qquad A_F=(c_{\mathrm{exc}}^{-1}+653600)^{-1}.
\]
Thus \eqref{fano-input} is exactly the simplified form of \cite[Theorem 1.1]{Fano}. We will also use the weighted fibre estimate of \cite[Proposition 4.1]{Fano}, stated below in \eqref{potential}, and its local matrix formulas \cite[Lemmas 2.2, 2.4 and 2.5]{Fano}. These previous arguments will not be repeated here.

For clarity, the local formulas used later are recorded together. Let $Q_p=-(E_i\cdot E_j)$, $D_p=\det Q_p$, and write $K_S+\sum_i(1-\alpha_i)E_i=f^*K_X$ locally on the minimal resolution. Then $\eps\le\alpha_i\le1$, and $D_p$ kills the local Weil class group. In the cyclic case, if the chain has length $s$, then
\begin{equation}\label{chain-det}
D_p=\sum_{i=0}^s\frac1{\alpha_i\alpha_{i+1}},\qquad\alpha_0=\alpha_{s+1}=1.
\end{equation}
If $l_i,r_i$ denote the determinants of the subchains strictly to the left and right of vertex $i$, with empty determinant one, then
\[
\alpha_i=\frac{l_i+r_i}{D_p},\qquad \beta_i:=(Q_p^{-1})_{ii}=\frac{l_ir_i}{D_p}.
\]
For a noncyclic point one has $D_p\le4/\eps$ and $\alpha_i\le\beta_i$ for every vertex. Consequently in all cases
\begin{equation}\label{alpha-beta}
\alpha_i\le\beta_i+\frac1{D_p}.
\end{equation}
The cyclic assertion gives $l_i+r_i\le l_ir_i+1$; the fork assertion is precisely \cite[Lemma 2.5]{Fano}. These are local results and do not assume a global Fano hypothesis.

The following constants will be used in the volume theorem:
\begin{equation}\label{more-constants}
\begin{aligned}
c_b&=\min\left\{\frac{A_F}{324},\frac1{1728}\right\},
&c_1&=\frac1{648},\\
c_0&=(344\cdot66^3)^{-18\cdot66^3},
&c&=\min\{c_b,c_1,c_0\}.
\end{aligned}
\end{equation}
The small constant $c_0$ is needed only when the anticanonical
nef model has numerically trivial canonical class.
If $K_X\equiv0$ itself, Proposition~\ref{zero-absolute}
gives the larger absolute gap $1/55440$.
The size of $c_0$ has no bearing on the powers of $\eps$.

\subsection{Zariski decomposition}
For a pseudo-effective $\Q$-divisor $D$ on a smooth projective surface, its Zariski decomposition is the unique expression $D=P_D+N_D$ with $P_D$ nef, $N_D$ effective with negative definite intersection matrix if nonzero, and $P_D$ orthogonal to every component of $N_D$. The coefficients are rational, and $\vol(D)=P_D^2$ when $D$ is big. We use the corresponding decomposition on a $\Q$-factorial surface, obtained by numerical pullback and descent. Detailed discussion can be found in Zariski \cite{Zariski}, Fujita \cite{Fujita}, and \cite[Section 2.3]{Lazarsfeld}.

\begin{lemma}\label{positive-part}
Let $X$ be a projective klt surface with $-K_X$ \peff, and write its Zariski decomposition as
\[
-K_X=P+N,\qquad P\text{ nef},\quad N\ge0,\quad P\cdot N=0.
\]
Suppose $H$ is nef and $K_X+\lambda H$ is \peff, with $\lambda>0$. Then
\begin{equation}\label{pairings}
\lambda H^2\ge H\cdot P,\qquad \lambda H\cdot P\ge P^2,
\end{equation}
and hence
\begin{equation}\label{zariski-volume}
H^2\ge\frac{P^2}{\lambda^2}.
\end{equation}
If $-K_X$ is big, $P^2=\vol(-K_X)$.
\end{lemma}
\begin{proof}
The divisor $\lambda H-P=(K_X+\lambda H)+N$ is \peff. Intersect it with the two nef divisors $H$ and $P$. This gives \eqref{pairings}, and then \eqref{zariski-volume}.
\end{proof}

\subsection{The ample model and the anticanonical model}
\begin{lemma}\label{H-model}
Let $(X,B)$ be a klt log Calabi--Yau pair and $H$ nef and big. Then $H$ is semiample. Its ample model $u:X\to\bar X$ satisfies
\[
H=u^*\bar H,\qquad H^2=\bar H^2,\qquad
K_X+B=u^*(K_{\bar X}+\bar B),
\]
where $\bar H=u_*H$ is ample and integral, $\bar B=u_*B\ge0$, and the pair discrepancy bound is preserved.
\end{lemma}
\begin{proof}
Apply the base-point-free theorem to a Cartier multiple of $H$, since its difference from $K_X+B$ is nef and big. The contraction is birational because $H$ is big. The descended $\Q$-Cartier ample divisor is its pushforward. Choose a small rational $t>0$ and a general effective $D\sim_\Q H$ such that $(X,B+tD)$ is klt. The divisor $K_X+B+tD\sim_\Q tH$ is semiample and big; thus $\bar X$ is its log canonical model and $(\bar X,\bar B+t\bar D)$ is klt. In particular $K_{\bar X}+\bar B$ is $\Q$-Cartier. The difference $K_X+B-u^*(K_{\bar X}+\bar B)$ is exceptional and numerically trivial over $\bar X$, so it vanishes by negativity. The canonical formula shows equality of all log discrepancies.
\end{proof}

This reduction preserves the volume of the polarisation, even if some contracted curves are $K_X$-positive. The crepant identity preserves the discrepancy bound for the pair.

For the anticanonical divisor we use the following MMP reduction.

\begin{lemma}\label{anti-model}
Let $(X,B)$ be a klt log Calabi--Yau pair. There is a birational morphism $g:X\to Y$ obtained by a $(-K_X)$-MMP such that $-K_Y$ is nef and semiample. The pair $(Y,B_Y)$ is crepant to $(X,B)$. If
\[
-K_X=P+N
\]
is the Zariski decomposition, then $P=g^*(-K_Y)$ and $N$ is effective and exceptional. If $B$ is big, $Y$ is weak Fano and
\begin{equation}\label{anti-volume}
\vol(-K_X)=(-K_Y)^2.
\end{equation}
\end{lemma}
\begin{proof}
Choose a small rational $\delta>0$ with $(X,(1+\delta)B)$ klt. The divisor
\[
K_X+(1+\delta)B\sim_\Q\delta B\sim_\Q-\delta K_X
\]
is \peff. Its surface MMP ends at a nef model, and log abundance makes the resulting divisor semiample. Thus $-K_Y$ is nef and semiample. Every step is crepant for $K+B\sim_\Q0$, so the original discrepancy bound is preserved.

Negativity for this MMP gives $-K_X=g^*(-K_Y)+N$ with $N\ge0$ exceptional. The first term is nef, is orthogonal to $N$, and the exceptional support is negative definite. This is the Zariski decomposition. Equation \eqref{anti-volume} follows.
\end{proof}

The auxiliary $\delta$ need not have an explicit bound. It is used to construct the MMP; the gap retained throughout is that of the original pair $(X,B)$.

\subsection{Determinants on rational log Calabi--Yau surfaces}\label{mass-sec}
Let $(X,B)$ be a rational $\eps$-lc log Calabi--Yau pair and let $f:S\to X$ be the minimal resolution. Write
\begin{equation}\label{resolution}
K_S+\Delta=f^*K_X,\quad
\Delta=\sum(1-\alpha_i)E_i,\quad
K_S+B_S=f^*(K_X+B)\sim_\Q0.
\end{equation}
Then $\Delta\ge0$ and $B_S=f^*B+\Delta\ge\Delta$. One can see the first effectivity directly: the exceptional curves are rational with weights $b_i=-E_i^2\ge2$, and their positive matrix satisfies $Q(\mathbf1-\alpha)=(b_i-2)_i\ge0$. Its inverse is entry-wise nonnegative. Thus the ambient discrepancies $\alpha_i$ lie in $[\eps,1]$, while the pair discrepancies $a_i=1-\operatorname{coeff}_{E_i}B_S$ satisfy
\begin{equation}\label{a-alpha}
\eps\le a_i\le\alpha_i\le1.
\end{equation}
This inequality is what allows the former weak Fano argument to extend.

If $S\ne\PP^2$, choose a birational morphism to a Hirzebruch surface. Its induced ruling has reduced fibres with simple normal crossings. The weighted potential in \cite[Proposition 4.1]{Fano} is
\[
\Phi=\sum_{V\cap W\text{ a reduced-fibre node}}\frac{1}{a_Va_W},
\qquad a_V=1-\operatorname{coeff}_V B_S.
\]
The previous estimate, valid for every smooth rational $\eps$-lc log Calabi--Yau pair, gives
\begin{equation}\label{potential}
\rho(S)-2\le\Phi\le4\eps^{-2}+12\eps^{-3},\qquad \rho(S)\le18\eps^{-3}.
\end{equation}
We next extend the determinant estimate from weak Fano surfaces to rational log Calabi--Yau pairs. The distinction is that a nonexceptional vertical component may now occur in $B_S$; the inequality $a_V\le1$ suffices in place of equality.

\begin{proposition}\label{mass}
For every rational $\eps$-lc log Calabi--Yau pair,
\begin{equation}\label{mass-bound}
\sum_{p\text{ noncanonical cyclic}}D_p\le344\eps^{-3},
\qquad D_p\le R(\eps)\text{ for every singular point }p.
\end{equation}
\end{proposition}
\begin{proof}
Let $n$ be the number of exceptional curves and let $I$ be the sum of $1/(\alpha_i\alpha_j)$ over internal edges of noncanonical cyclic chains. Vertical edges are bounded by $\Phi$, by \eqref{a-alpha}. Edges with a horizontal endpoint and both ambient discrepancies at least $1/2$ contribute at most $4n$.

For the other edges, the discrepancy equation gives $2\alpha_i\le1+\alpha_j$ for adjacent vertices. A horizontal endpoint adjacent to a discrepancy below $1/2$ must itself have discrepancy below $3/4$. Since
\[
\sum_{E_i\text{ horizontal}}(1-\alpha_i)\le\Delta\cdot F\le B_S\cdot F=2,
\]
there are fewer than eight such horizontal curves. Each has at most two neighbours in its cyclic chain. Thus
\[
I\le\Phi+4n+16\eps^{-2}.
\]
For a chain of length at least two, each endpoint term in \eqref{chain-det} is at most its adjacent internal-edge term. These chains contribute at most $3I$.

An isolated vertical exceptional curve has $D_p=2/\alpha_i$. Its fibre is reducible, since a fibre with only one component has square zero, whereas an exceptional curve has negative self-intersection. Connectedness supplies a neighbouring vertical curve, which cannot be exceptional for $f$: intersecting $f$-exceptional curves map to the same point, contrary to isolation. Its contribution is at most twice the corresponding pair-potential edge; distinct isolated curves choose distinct edges. These terms total at most $2\Phi$. An isolated horizontal noncanonical curve has $1-\alpha_i\ge1/3$ by the singularity classification, so there are at most six, each contributing at most $2/\eps$.

Since $n\le1+\Phi$, the total is at most
\[
3I+2\Phi+12\eps^{-1}
\le17\Phi+12+48\eps^{-2}+12\eps^{-1}
\le344\eps^{-3}.
\]
A canonical cyclic chain has determinant $s+1\le\rho(S)\le18\eps^{-3}$. The fork estimate handles noncyclic points. The case $S=\PP^2$ is immediate.
\end{proof}

The extension in Proposition~\ref{mass} is useful precisely because the ambient anticanonical divisor need not be nef. It bounds individual local denominators on every rational model crepant for the log Calabi--Yau pair. It does \emph{not} replace an arbitrary lcm by a number at most $R(\eps)$.

\subsection{Linear resolution bounds and local determinants}\label{r1-local}
Let $f:S\to X$ be the minimal resolution of a rank-one $\eps$-lc Fano
surface. The surface $S$ is rational. Write
\[
 f^*K_X=K_S+\sum_i(1-\alpha_i)E_i,
 \quad b_i=-E_i^2\geq2,\quad M=(-E_i\cdot E_j),\quad
 \eps\leq\alpha_i\leq1.
\]
For each singular point $p$, let $s_p$ be its resolution length and
$D_p=\det M_p$. The local class group has order $D_p$; its exponent
divides $D_p$. In particular $\lcm_p D_p$ makes every Weil divisor Cartier.
For a noncyclic quotient, $D_p$ need not equal the order of the local
fundamental group.

\begin{lemma}\label{r1-lem:local}
In this setting,
\begin{align}
 \rho(S)&<10+\frac{12(1-\eps)^2}{\eps}\leq\frac{12}{\eps},\label{r1-eq:rho}\\
 D_p&\leq12\eps^{-3}\quad\text{for every }p,\label{r1-eq:Dall}\\
 D_p&\leq4\eps^{-1}\quad\text{for noncyclic }p,\label{r1-eq:Dfork}\\
 D_p&\leq12\eps^{-1}\quad\text{for canonical }p,\label{r1-eq:Dcan}\\
 \sum_pD_p&\leq15\eps^{-3}.\label{r1-eq:Dsum}
\end{align}
\end{lemma}

\begin{proof}
The discrepancy and volume identities are
\begin{equation}\label{r1-eq:local-identities}
 M(1-\alpha)=(b_i-2)_i,\qquad
 V:=(-K_X)^2=10-\rho(S)+\sum_p\delta_p,
 \quad\delta_p=\sum_{i\in p}(b_i-2)(1-\alpha_i).
\end{equation}
If $v_i$ denotes valency, summing the discrepancy equations gives
\[
 \sum_{i\in p}(b_i-2)\alpha_i
   =\sum_{i\in p}(2-v_i)(1-\alpha_i).
\]
For a chain, the right side is at most $2(1-\eps)$; this includes a
one-vertex chain. For a fork it is at most $3(1-\eps)$, since the
central term is nonpositive. Hence
\[
 \delta_p\leq\frac{2(1-\eps)^2}{\eps}\text{ for a chain},\qquad
 \delta_p\leq\frac{3(1-\eps)^2}{\eps}\text{ for a fork}.
\]
Belousov's theorem \cite[Theorem 1.1]{Belousov} gives at most four singular points. Since $V>0$,
\eqref{r1-eq:rho} follows. The final inequality is equivalent to
$\eps(12\eps-14)\leq0$.

For a cyclic chain put $\alpha_0=\alpha_{s_p+1}=1$. The continuant identity \eqref{chain-det} is
\begin{equation}\label{r1-eq:wronskian}
 D_p=\sum_{i=0}^{s_p}\frac1{\alpha_i\alpha_{i+1}}
 \leq(s_p+1)\eps^{-2}.
\end{equation}
Since $s_p+1\leq\rho(S)$, this proves \eqref{r1-eq:Dall} for chains.
For completeness, the fork Schur-complement formula is
\[
 D_p\alpha_0=\prod_{j=1}^3r_j
 \left(\sum_{j=1}^3\frac1{r_j}-1\right),
\]
where $r_j$ are the three arm determinants and $\alpha_0$ is the central
discrepancy. The spherical triples $(2,2,r)$, $(2,3,3)$, $(2,3,4)$,
$(2,3,5)$ give respectively $4,3,2,1$ on the right. This proves
\eqref{r1-eq:Dfork}, and therefore \eqref{r1-eq:Dall}. For a canonical chain,
$D_p=s_p+1\leq\rho(S)$; canonical forks have determinant at most four.
This proves \eqref{r1-eq:Dcan}.

Finally let $c$ and $f$ be the numbers of chains and forks, and let $S_c$
be the total number of chain vertices. A fork has at least four vertices,
so $S_c+c\leq\rho(S)-1-4f+c\leq\rho(S)+3-5f$. Using $\eps\leq1$,
\[
 \sum_pD_p\leq\frac{S_c+c}{\eps^2}+\frac{4f}{\eps}
 \leq\frac{\rho(S)+3-f}{\eps^2}\leq15\eps^{-3}.
\]
\end{proof}

\section{Rank-one anticanonical volumes}\label{rank-one-sec}
We prove Theorem~\ref{main-rank-fano} by retaining the relations
between singularities in the classification. A basket records
the full resolution chains and forks. For an admissible chain
$T$, write $d(T)$ for its determinant and $\two{\ell}$ for
a string of $\ell$ entries equal to two; a string of length
zero is omitted. The height is the least intersection of the
reduced exceptional divisor with a general fibre of a
$\PP^1$-fibration on the minimal resolution
\cite[Definition 1.1]{PP12}.

\begin{table}[htbp]
\centering\small
\caption{Notation for the external classification lists.}
\label{tab:classification-notation}
\begin{tabular}{@{}ll@{}}
\toprule
Notation used here & Source\\\midrule
$9(j),10(j),11(j)$ & The indicated cases in Tables 9, 10, 11 of \cite{PP12}\\
$A_j$ & Case $(j)$ of \cite[Lemma 4.12]{PP3}\\
$B_j$ & Case $(j)$ of \cite[Lemma 5.15]{PP3}\\
LDP$j$ & Series $j$ in the published classification \cite{Lacini}\\
\bottomrule
\end{tabular}
\end{table}

\begin{proposition}[Classification reduction]\label{classification-reduction}
Let $X$ be a complex klt Fano surface with $\rho(X)=1$.
If $X$ is not canonical, at least one of the following applies:
\begin{enumerate}[label=(\roman*)]
\item Its minimal resolution has height at most two, and its
basket occurs among the klt entries of \cite[Tables 9--11]{PP12},
including the two final cases of Table 10.
\item It has a descendant with nodal elliptic boundary in the
sense of \cite[Definition 1.11]{PP12}; in this case
Lemma~\ref{r1-nodal} applies.
\item Its basket occurs in the remaining lists recorded in
Appendix~\ref{classification-app}: the parameter families
of Tables~\ref{tab:cyclic-rays} and \ref{tab:fork-rays},
the finite baskets of Table~\ref{tab:finite-baskets},
or the finite part of $A_{23}$ in Table~\ref{tab:A23}.
\end{enumerate}
The extra cases $A_{16},B_2$ and the possible omission in
\cite[Example 3.13]{Nagaoka} are included.
\end{proposition}
\begin{proof}
We use the published 24-series classification of
\cite{Lacini}, together with its characteristic-zero
correspondence in \cite[Section 7, Tables 3--4]{PP3}.
Table~\ref{tab:lacini-crosswalk} lists the destination of
every published series. The low-height destinations are
precisely those of \cite[Theorem A]{PP12}. For the
elliptic-descendant branch outside height at most two,
\cite[Theorem E(e)]{PP12} makes the boundary nodal.
The separate constructions from
\cite[Proposition 5.1(3)--(5)]{Lacini} are included in
the remaining parameter families, including the two
$4A_2$-seed branches described in \cite[Section 7]{PP3}.
The characteristic-five series is absent over $\C$.
The additional $A_{16},B_2$ baskets in \cite{PP3} and
the possibly omitted basket identified in \cite{Nagaoka}
are explicitly retained in the finite table.
This establishes coverage by the stated lists.
Only their displayed constructions and full resolution
baskets are used; the exact heights left to subsequent
work in \cite{PP3} are not needed.
\end{proof}

The classification supplies the complete parameter ranges.
The calculations below and in Appendix~\ref{classification-app}
supply the volume estimates on those ranges.
\subsection{Intersection identities and fork determinants}
Let $f:S\to X$ be the minimal resolution, with $s$ exceptional curves. For a connected exceptional graph set
\[
 M=(-E_iE_j),\quad b_i=-E_i^2,\quad
 f^*K_X=K_S+\sum_i(1-\alpha_i)E_i,
\]
and write $D_p=\det M_p$, $\delta_p=\sum_i(b_i-2)(1-\alpha_i)$. Rationality and $\rho(X)=1$ give
\begin{equation}\label{r1-eq:volume}
 V:=K_X^2=9-s+\sum_p\delta_p,\qquad
 V\prod_{p\text{ noncanonical}}D_p\in\ZZ_{>0}.
\end{equation}
Indeed, $M(1-\alpha)=(b_i-2)_i$, so $D_p\delta_p$ is integral; canonical points contribute zero. Also
\begin{equation}\label{r1-eq:diagonal}
 b_i\eps\le2.
\end{equation}
This follows from $b_i\alpha_i=2-\operatorname{val}(i)+\sum_{j\sim i}\alpha_j\le2$ on a minimal quotient resolution.

For a fork with central weight $b$ and arm determinants $d_1,d_2,d_3$, let $q_j$ be the cofactor at the end meeting the centre. Eliminating the arms gives
\begin{equation}\label{r1-eq:fork}
 \alpha_0=\frac{c}{b-\sum q_j/d_j},\qquad
 c=\sum_j\frac1{d_j}-1,\qquad
 D_F\alpha_0=h_F:=d_1d_2d_3c.
\end{equation}
The klt condition is $c>0$. The spherical triples give $h_F=4,3,2,1$ for $(2,2,r)$, $(2,3,3)$, $(2,3,4)$, $(2,3,5)$ respectively. In particular
\begin{equation}\label{r1-eq:forkbound}
 D_F\le h_F/\eps\le4/\eps.
\end{equation}
Since $q_j\le d_j-1$ and $c\le1/2$,
\[
 \alpha_0\le\frac{c}{b-2+c}\le\frac1{2b-3}.
\]
Thus $b\eps\le1$ whenever $b\ge3$. Keeping $h_F$ and this stronger central-weight estimate materially improves the constants.

The one-point estimates of \cite[Theorem 8.2 and Corollary 8.4]{Fano}
give $V\ge\eps^2/2$ if there is at most one noncanonical point, which is cyclic,
and $V\ge\eps^2/4$ if there is at most one noncanonical point of any type.
We also use the toric estimate $V\ge\eps^2$, with no Picard number
restriction, from \cite[Proposition 8.9]{Fano}. Canonical surfaces
have $V\in\Z_{>0}$. These are the only anticanonical volume inputs
imported in the rank-one proof.

\subsection{Elliptic descendants}
\begin{lemma}[The elliptic-descendant branch]\label{r1-nodal}
Suppose that a rank-one klt del Pezzo surface has a descendant
$(Y,T)$ with $Y$ canonical and $T\subset Y_{\rm reg}$ a nodal
anticanonical curve, in the sense of \cite[Definition 1.11]{PP12}.
Then it has at most one noncanonical point, and that point is cyclic.
\end{lemma}
\begin{proof}
The construction is an isomorphism near the ADE configurations of $Y$.
Over the node, \cite[Lemma 1.12 and Lemma 6.2(a)]{PP12} leaves exactly
one exceptional $(-1)$-curve uncontracted to $X$. Removing that curve
from the total-transform cycle leaves one admissible chain.
An outer blowup would leave either a cycle in the exceptional locus
of a quotient point or an additional $(-1)$-curve in the minimal
resolution; neither is possible. Thus all new contracted curves
belong to one cyclic configuration, and the other points remain
canonical.
\end{proof}
By \cite[Theorem E(e)]{PP12}, over $\C$ the elliptic-descendant
branch outside height at most two has nodal boundary. The hypothesis
$T\subset Y_{\rm reg}$ is essential here: the separate inverse-blowup
constructions from \cite[Proposition 5.1(3)--(5)]{Lacini}
belong to the remaining case analysis below.
\subsection{Continuants and the coupled cyclic row $10(2)$}
For an oriented admissible chain $T=[b_1,\ldots,b_s]$, write
\[
 d=d(T),\qquad p=d([b_1,\ldots,b_{s-1}]),\qquad
 q=d([b_2,\ldots,b_s]).
\]
The empty determinant is one. The matrix identity
\begin{equation}\label{r1-eq:continuant}
 M(T):=\prod_{i=1}^s\begin{pmatrix}b_i&-1\\1&0\end{pmatrix}
 =\begin{pmatrix}d&-p\\q&-(pq-1)/d\end{pmatrix}
\end{equation}
fixes all orientation conventions. The adjoint $T^*$ is characterised by the contractible fibre $[T,1,T^*]$; its data are
\begin{equation}\label{r1-eq:adjoint}
 d^*=d,\qquad p^*=d-q,\qquad q^*=d-p.
\end{equation}
The operation $A*B$ replaces the adjoining end weights $a,b$ by $a+b-1$. Reversing a chain interchanges $p$ and $q$.

We shall repeatedly use the following elementary version of the cyclic age formula. For the cyclic singularity represented by $T$, if $x,y$ are positive integers and
\begin{equation}\label{r1-eq:age}
 y\equiv px\pmod d,
\end{equation}
then its minimal log discrepancy is at most $(x+y)/d$. To see this, if $x+y\ge d$ use $\mld\le1$. Otherwise $0<x,y<d$, and the primitive quotient-lattice vector on the ray through $(x/d,y/d)$ gives a toric divisorial valuation of log discrepancy at most $(x+y)/d$. Using $q$ instead of $p$ gives the isomorphic reversed quotient.

\begin{proposition}\label{r1-prop:row2}
In row $10(2)$, $V\ge(r-1)\eps^2\ge\eps^2$, with no restriction on the lengths of $T,U,W$.
\end{proposition}
\begin{proof}
Set
\[
\begin{gathered}
 a=d(T),\quad p=d(T\text{ minus its last component}),\\
 b=d(U),\quad v=d(U\text{ minus its first component}),\\
 c=d(W),\quad w=d(W\text{ minus its last component}),\quad k=r-1,\\
 x=ab-pv,\qquad y=bw+v(c-w),\qquad d_0=a(c-w)+pw.
\end{gathered}
\]
These numbers are positive. Let $C=[T,U,r,U^*,W^*]$ and $E=[W,T^*]$. Multiplying \eqref{r1-eq:continuant} using \eqref{r1-eq:adjoint} gives
\begin{equation}\label{r1-eq:row2id}
 d(E)=d_0,\qquad d(C)=D=kxy-d_0,\qquad
 V=\frac{k(x+y)^2}{d_0D}.
\end{equation}
To verify the volume identity, recall that for a chain with end data $(d,p,q)$, the two endpoint discrepancies are $(q+1)/d$ and $(p+1)/d$. Summing the discrepancy equations gives
\[
 \delta(T)-\#T=\sum_{i=1}^s(b_i-3)-2+
 \frac{p+q+2}{d}.
\]
The adjoint identity $\sum_T(b_i-3)+\sum_{T^*}(b_i-3)=-2$, together with the canonical chain of length $r-2$, reduces $9-s+\sum\delta$ to
\[
 \frac{p_C+q_C+2}{D}+\frac{p_E+q_E+2}{d_0}-2.
\]
Substitution of the two matrix products gives precisely \eqref{r1-eq:row2id}.

The left and right subchains at the inserted $r$-vertex have determinants $x$ and $y$. Its discrepancy is
\[
                         \alpha_r=\frac{x+y}{D}.
\]
For the other singularity, the two vectors $(a,w)$ and $(-p,c-w)$ form a basis of its integral age lattice: their determinant is $d_0$, and the cofactor in \eqref{r1-eq:continuant} verifies the congruence \eqref{r1-eq:age} for each. Since
\[
                    (x,y)=b(a,w)+v(-p,c-w),
\]
\eqref{r1-eq:age} gives $\mld(E)\le(x+y)/d_0$. Hence
\[
 \eps\le\frac{x+y}{\max\{d_0,D\}},\qquad
 V=\frac{k(x+y)^2}{d_0D}
 \ge k\frac{(x+y)^2}{\max\{d_0,D\}^2}
 \ge k\eps^2.
\]

\end{proof}

\subsection{A two-section calculation}
The next formula is useful because it preserves the relation between the singularities. In the two-section cases of \cite{PP12}, contract the vertical exceptional divisor of the minimal resolution while retaining the two horizontal exceptional sections. Denote the resulting surface by $T$ and its sections by $E_1,E_2$. The contraction $f:T\to X$ has relative Picard number two, so $\rho(T)=3$. For a general fibre $F$, put
\[
 E_i^2=-e_i<0,\quad E_1E_2=0,\quad E_iF=1,\quad F^2=0,
 \qquad a_i=a(E_i,X).
\]
The three classes are a basis: their intersection determinant is $e_1+e_2>0$. Adjunction along a section gives
\begin{equation}\label{r1-eq:sectiondelta}
 \delta_i:=a_ie_i=2-v_i+\sum_{j=1}^{v_i}\frac1{d_{ij}},
\end{equation}
where $v_i$ is the number of attached exceptional twigs and $d_{ij}$ their determinants. This follows either from the different on the section or from eliminating each twig in the discrepancy equations.

\begin{lemma}\label{r1-lem:sections}
In this situation,
\begin{equation}\label{r1-eq:sections}
                 V=\frac{(a_1+a_2)^2}{e_1^{-1}+e_2^{-1}}.
\end{equation}
If $\delta_1=\delta_2$, then $V=a_1a_2(e_1+e_2)$.
\end{lemma}
\begin{proof}
Write $L=f^*(-K_X)=xE_1+yE_2+zF$. The equations $LE_i=0$ give $x=z/e_1$ and $y=z/e_2$. The crepant formula intersected with $F$ gives $LF=a_1+a_2$. Thus $z=(a_1+a_2)/(e_1^{-1}+e_2^{-1})$ and $L^2=z(a_1+a_2)$, proving the first assertion. For the second, use $a_1e_1=a_2e_2$.
\end{proof}

For example, in row $10(11)$, the two columnar fibres have complementary end cofactors. They give equal $\delta_i$ and $e_1+e_2=n+m-2\ge2$. Therefore $V\ge2\eps^2$.

\subsection{One fork and two cyclic points: row $10(15)$}
\begin{proposition}\label{r1-prop:row15}
Every klt surface in row $10(15)$ satisfies $V\ge\eps^2/8$.
\end{proposition}
\begin{proof}
One horizontal section is the centre of a fork with arm determinants $d_1,d_2,d_3$; the other meets the two complementary columnar arms. Thus
\[
 \delta_1=-1+\frac1{d_1}+\frac1{d_2}+\frac1{d_3}>0,
 \qquad \delta_2=\frac1{d_1}+\frac1{d_2}>\frac12.
\]
If the unordered triple is not dihedral, it is $(2,3,3)$, $(2,3,4)$, or $(2,3,5)$. Hence $\delta_1\ge1/30$. Since $a_i\le1$, $e_i\ge\delta_i$, and \eqref{r1-eq:sections} gives
\[
                         V\ge\frac{4\eps^2}{30+2}=\frac{\eps^2}{8}.
\]

Suppose the triple is $(2,2,r)$. Write $n$ for the central weight of the fork and $q/r$ for the contribution of its variable arm at the centre. If $n\ge3$, then $e_1=n-1-q/r>1$ and $e_2>1/2$. Formula \eqref{r1-eq:sections} even gives $V\ge4\eps^2/3$.

It remains to treat $n=2$. We orient every fork arm \emph{from its centre to its tip}; write $T$ for the variable arm, $r=d(T)$, $q=d(T\text{ minus its first component})$, and $A=r-q$. The central discrepancy is $1/A$. There are two positions for this arm.

\emph{The variable arm is in the special fibre.} Put $h=m-1\ge1$. The two cyclic graphs, in the orientation just specified, are
\[
                         [2,m,2],\qquad C=T^**[\two h].
\]
Here $\delta_1=1/r$, $e_1=A/r$, $a_1=1/A$, while $\delta_2=1$, $e_2=h$, and $a_2=1/h$. Thus Lemma~\ref{r1-lem:sections} and the continuant give
\begin{equation}\label{r1-eq:row15a}
 D=hr+A,\qquad V=\frac{(A+h)^2}{AhD}.
\end{equation}
The end cofactor $p_C=D-r$ satisfies $hp_C\equiv A\pmod D$. Thus $(h,A)$ is an age vector for $C$, and
\[
                 \eps\le\frac1A,\qquad \eps\le\frac{A+h}{D}.
\]
Multiplication gives $\eps^2\le(A+h)/(AD)\le V$.

\emph{The variable arm is columnar.} Put $h=2m-1\ge3$. The cyclic graphs are
\[
                         C=[T^*,m,2],\qquad [3,\two{m-2}].
\]
Here $e_2=m-1/2-A/r$, $\delta_2=1/2+1/r$. Substitution in Lemma~\ref{r1-lem:sections} gives
\begin{equation}\label{r1-eq:row15b}
 D=hr-2A,\qquad V=\frac{(A+h)^2}{AhD},\qquad p_C=\frac{D+r}{2}.
\end{equation}
Since $h$ is odd, $hp_C\equiv A\pmod D$ follows from $D=hr-2A$. The same two discrepancy bounds prove $V\ge\eps^2$.

All identities in \eqref{r1-eq:row15a}--\eqref{r1-eq:row15b} follow by \eqref{r1-eq:continuant} and the Schur complement at the fork centre. The specified orientations can also be read directly from the fibre construction in the proof of \cite[Lemma 5.10]{PP12}: the columnar fibres are $[T_j,1,T_j^*]$; the special chain is $[T,1,T^*]*[\two{n+m-3},1]$. Consequently these normal forms do not depend on translating a transposed, unoriented fork symbol in isolation.
\end{proof}

\subsection{Two forks and a cyclic point: row $10(20)$}
\begin{proposition}\label{r1-prop:row20}
Every klt surface in row $10(20)$ satisfies $V\ge\eps^2/15$.
\end{proposition}
\begin{proof}
The two forks share two columnar fibres, with arm determinants $d_1,d_2$; let the remaining arm determinants be $d_3,d_4$. Both triples $(d_1,d_2,d_3)$ and $(d_1,d_2,d_4)$ are spherical.

If $(d_1,d_2)\ne(2,2)$ and $\max(d_1,d_2)\le5$, every $d_i\le5$ and $\delta_i\ge1/30$. Therefore \eqref{r1-eq:sections} gives $V\ge4\eps^2/60=\eps^2/15$.

If $\max(d_1,d_2)=r>5$, sphericity forces the other common determinant, $d_3$, and $d_4$ all to be two. The complementary cofactors in the variable columnar fibre give
\[
 \delta_1=\delta_2=1/r,\qquad e_1+e_2=n+m-3\ge1.
\]
Lemma~\ref{r1-lem:sections} gives $V\ge\eps^2$.

Finally suppose $d_1=d_2=2$. Here the two remaining arms can both be arbitrarily long. We use the global fibre geometry. On the smooth resolution the two common columnar fibres are
\[
                 [2,1,2],\qquad\text{with multiplicities }(1,2,1).
\]
Take the double cover of the base $\PP^1$ branched at these two fibres, and normalize the base change. Equivalently, adjoin the square root of a rational function on the base with a simple zero and pole at these two points. Along divisors on the resolution its only odd valuations occur on the four end $(-2)$-curves. These four curves are exceptional over $X$. The induced finite map of normal surfaces
\[
                              \pi:Y\longrightarrow X
\]
is therefore quasi-\'etale of degree two.

For completeness, the resulting resolution is toric. Above each branch fibre, normalisation gives the smooth chain $[1,2,1]$, all of multiplicity one. This can be checked locally from $t=xy^2$: after adjoining $\sqrt t$, normalisation is given by $x=z^2$. Globally the ramified end curves have square $-1$ and the middle curve has square $-2$. Contract both end curves to obtain a smooth fibre. These contractions factor through $Y$, because the end curves map to points of $Y$.

The two sections remain disjoint sections after base change. The remaining special fibre splits into two chain fibres, each meeting the sections at its tips. We now have a smooth ruled surface with two disjoint sections and at most two reducible fibres, all chains. Contract vertical $(-1)$-curves until reaching a Hirzebruch surface. An internal contraction is the inverse of a blowup at a node of two vertical components; a tip contraction is the inverse of a blowup at a section--fibre node. The two sections and the two distinguished fibres on the Hirzebruch surface are a toric boundary. Reversing the contractions therefore gives a toric resolution. Every curve contracted from this resolution to $Y$ belongs to that boundary, so $Y$ is toric.

Since $K_Y=\pi^*K_X$, $Y$ is Fano and $K_Y^2=2K_X^2$. For a valuation upstairs, its log discrepancy is the ramification index times that of the restricted valuation downstairs. Thus $Y$ is still $\eps$-lc. Applying the toric theorem from \cite{Fano}, which does not require $\rho(Y)=1$, gives
\[
                          2V=K_Y^2\ge\eps^2.
\]
This proves the proposition.
\end{proof}

\subsection{The height-one section calculation}
We treat the height-one rows with arbitrary chains. Use the ruling in
\cite[Lemma 4.4]{PP12}. Contract its vertical exceptional curves
but retain the exceptional section $E$, and write
\[
 E^2=-e,\quad EF=1,\quad F^2=0,\quad
 e=m-\sum_{j=1}^{\nu}q_j/d_j,\qquad
 \delta=ae=2-\nu+\sum_{j=1}^{\nu}1/d_j,
\]
where $a=a(E,X)$, $m$ is its original weight, and $q_j/d_j$
are the twig contributions at $E$. The remaining numerical space
has rank two. For $L=f^*(-K_X)$, $LE=0$ and $LF=1+a$.
Solving these two equations gives
\begin{equation}\label{r1-heightone}
              V=(1+a)^2e=\frac{(e+\delta)^2}{e}.
\end{equation}
In a columnar fibre $[T,1,T^*]$, the opposite endpoint has
discrepancy $u=1-q/d+1/d$. If $\nu\le2$ and $m\ge2$,
then $e+\delta=m+2-\nu-\sum q_j/d_j+\sum1/d_j\ge u$.
Also $\delta>0$, so $V\ge e+\delta\ge u\ge\eps$.
This treats rows $9(2),9(3),9(5)$; row $9(1)$ is
$\PP(1,1,m)$ and has $V=(m+2)^2/m\ge8$.

In row $9(7)$, the three determinants form a spherical triple.
Outside the dihedral case, $\delta\ge1/30$, and
\eqref{r1-heightone} gives $V\ge4\delta\ge2/15$.
For $(2,2,r)$ let $r=d(T)$, and let $q$ be the determinant
of $T$ with the vertex adjacent to $E$ removed.
Then $\delta=1/r$ and $e=m-1-q/r$. If $m\ge3$,
$e>1$. If $m=2$, put $A=r-q\ge1$; then
\[
             V=\frac{(A+1)^2}{Ar}\ge\frac{A+1}{r}\ge\eps,
\]
where the last term is the relevant endpoint discrepancy of $T^*$.
This proves $V\ge\min\{\eps,2/15\}$ in the whole row.
\subsection{The arbitrary-chain row $9(18)$}
We supply a direct estimate, avoiding an orientation-sensitive intermediate determinant bound. The \cite{PP12} fork is written tip-to-centre as
\[
 \langle b;[T,m,2],[2],[2]\rangle+T^*
       +[3,\two{b-3}],
\]
where the last chain is omitted for $b=2$. Let $r=d(T)$, and let $q$ be the determinant of $T$ with its last vertex removed. Contract the vertical exceptional curves while retaining the horizontal curve $E$ of weight $m$ and a general fibre $F$. The resulting rank-two intersection space has
\[
 E^2=-e,\quad EF=1,\quad F^2=0,\qquad
 e=m-\frac qr-\frac{b-1}{2b-3},\qquad
 \alpha_E=\frac1{re}.
\]
These formulas are the Schur complement at $E$; its two attached pieces contribute $q/r$ and $(b-1)/(2b-3)$, and the numerator of its discrepancy equation is $1/r$.

For $L=f^*(-K_X)$, one has $LE=0$ and $LF=1+\alpha_E$, because $E$ is the unique exceptional horizontal section. Solving in the basis $E,F$ gives
\begin{equation}\label{r1-eq:row18}
                       V=(1+\alpha_E)^2e.
\end{equation}
If $m\ge3$, then $e>1$. If $b\ge3$, then $e>1/3$. In the remaining case $b=m=2$, put $A=r-q\ge1$. Then
\[
 e=A/r,\qquad V=\frac{(A+1)^2}{Ar}\ge\frac{A+1}{r}\ge\eps.
\]
The last inequality uses the corresponding endpoint discrepancy of the adjoint chain $T^*$. Consequently this entire arbitrary-chain family satisfies
\[
                         V\ge\min\{1/3,\eps\}\ge\eps^2/3.
\]

\subsection{Completion of the quadratic estimate}
The geometric calculations above treat the arbitrary-chain
configurations. The other low-height cases are recorded,
with their determinant estimates, in Table~\ref{tab:low-height}.
Tables~\ref{tab:cyclic-rays}--\ref{tab:A23} treat the
remaining parameter families and finite baskets.
\begin{proof}[Proof of Theorem~\ref{main-rank-fano}]
If $X$ is canonical, $V\in\Z_{>0}$. Otherwise Proposition~\ref{classification-reduction} reduces the proof to height at most two, the
elliptic-descendant branch, or the baskets in Appendix~\ref{classification-app}.
Table~\ref{tab:low-height}, including the coupled rows proved above, gives $V\ge\eps^2/45$. Lemma~\ref{r1-nodal}
and the one-point estimate give $V\ge\eps^2/2$ on the remaining
elliptic descendants. For the remaining baskets, \eqref{r1-eq:volume} and
Tables~\ref{tab:cyclic-rays}--\ref{tab:A23} give $V\ge\eps^2/96$.
Thus $V\ge\eps^2/100$ in every case. Example~\ref{quadratic-example}
proves that the exponent two cannot be decreased.
\end{proof}

\section{Common Weil indices and rank-one polarisations}\label{weil-index-sec}
For an integral ample Weil divisor $H$, a common Cartier
multiple $q$ gives $qH^2\in\Z_{>0}$. We bound this common
multiple in seventh order. Lemma~\ref{r1-lem:local} controls
all baskets except for a triple of noncanonical cyclic
points; the classification identifies the latter with
toric baskets. Thus this argument needs only a small
part of the classification used in Section~\ref{rank-one-sec}.
\subsection{The toric determinant product}
We use one planar geometry-of-numbers input: a convex body in $\RR^2$
with at most one interior lattice point has lattice width at most three
\cite[Theorem 1.1]{ACFH}. The body need not be a lattice polygon.

\begin{proposition}\label{r1-prop:toric}
Let $X$ be a toric $\eps$-lc Fano surface with $\rho(X)=1$. If
$D_1,D_2,D_3$ are the determinants of its three affine toric charts,
including $D_i=1$ for a smooth chart, then
\begin{equation}\label{r1-eq:toric-product}
 D_1D_2D_3\leq162\eps^{-7}.
\end{equation}
Consequently $q=\lcm(D_1,D_2,D_3)\leq162\eps^{-7}$ makes every integral
Weil divisor Cartier, and $H^2\geq\eps^7/162$ for every integral ample $H$.
\end{proposition}

\begin{proof}
Work in the \emph{actual} fan lattice $N$, including for a fake weighted
projective plane. Let $P$ be the triangle of primitive fan rays.
The $\eps$-lc condition is precisely
\begin{equation}\label{r1-eq:eps-triangle}
 \Int(\eps P)\cap N=\{0\}.
\end{equation}
Flatness gives a primitive integral functional with width at most
$3/\eps$ on $P$. Make it the second coordinate by a linear unimodular
change of coordinates; no translation of the origin is used. Let $W$
be this vertical width. Then
\begin{equation}\label{r1-eq:width-and-section}
 W\leq3/\eps,\qquad
 P\cap\{y=0\}\subset[-1/\eps,1/\eps]\times\{0\}.
\end{equation}
The second assertion follows by applying \eqref{r1-eq:eps-triangle} to the
two primitive lattice vectors $(\pm1,0)$.

Suppose first that no vertex is horizontal. After reversing the vertical
coordinate, two vertices have positive heights $p,q$ and one has height
$-r$, where $p,q,r$ are positive integers. The two edges incident to the
negative vertex meet the horizontal axis at $s_1>0>s_2$. Two of the
positive chart determinants are
\[
 A=(p+r)s_1,\qquad B=-(q+r)s_2,
 \qquad A,B\leq W/\eps.
\]
The third is
\[
 C=\frac{qA+pB}{r}\leq\frac{2W^2}{r\eps}
 \leq\frac{2W^2}{\eps}.
\]
Therefore $ABC\leq2W^4/\eps^3\leq162\eps^{-7}$.

If a vertex is horizontal, it is $(1,0)$ or $(-1,0)$ by primitivity.
The other two heights are $p$ and $-r$. Two determinants are $p,r$;
the remaining determinant is at most $(p+r)/\eps=W/\eps$ by the
horizontal-section bound. Their product is at most
$W^3/\eps\leq27\eps^{-4}\leq162\eps^{-7}$.

Finally $qH$ is Cartier, so $(qH)\cdot H=qH^2$ is an integer: a
Cartier divisor has integral degree on every integral curve. It is
positive by ampleness. Thus $H^2\geq1/q$.
\end{proof}

The exponent is not dependent on the recently sharpened flatness constant.
For example, classical planar hollow-body flatness implies width less
than five for a body $K$ with the origin as its unique interior lattice
point. Choose a point $z\in K$ of maximal Euclidean norm. Then
$-z\notin\Int K$, so $(K+z)/2\subset K$ has no interior lattice
points. Classical flatness (Hurkens; see \cite[Theorem A]{ACFH}) gives
$w(K)/2\leq1+2/\sqrt3$, hence $w(K)<5$.
Using five instead of three above gives $2\cdot5^4=1250$ in place of
$162$. This still implies the global constant $1728$ below. Thus the
sharp exponent, and even that global constant, do not require the newer
one-point flatness theorem.

\subsection{The classification filter for arbitrary polarisations}
\begin{lemma}\label{r1-classification}
A complex rank-one klt Fano surface with exactly three singular
points, all cyclic and noncanonical, has the same full singularity
basket as a rank-one toric Fano surface.
\end{lemma}
\begin{proof}
Proposition~\ref{classification-reduction} and the tables in Appendix~\ref{classification-app} leave only row $4.4(3)$ of
\cite[Table 9]{PP12} and row $5.10(1)$ of
\cite[Table 10]{PP12}. Indeed, in Table 9 every other cyclic
triple has a canonical companion, while the remaining triples
have a fork or are non-klt. In Table 10, the cyclic rows
$5.10(2),(3),(11),(12),(22)$ have an explicit canonical
$A$-chain companion; all other triples have a fork or are
non-klt. Table 11 has a fork or a canonical companion in every
relevant triple. The higher-height correspondence has at most
two noncanonical points, and Lemma~\ref{r1-nodal} handles
the elliptic descendants. The additional $A_{16},B_2$ baskets
have a fork, and Nagaoka's possible extra basket has a canonical
third point, so these corrections introduce no further case.

The two surviving rows have toric realisations, as in
\cite[Lemmas 4.4, 5.10 and Section 7C]{PP12}. In row $4.4(3)$,
start with a Hirzebruch surface, its two disjoint sections,
and the two fibres containing the columnar chains. In row
$5.10(1)$ the corresponding construction is the toric
rank-one case described there. In both constructions the
modifications are blowups of boundary nodes and contractions
of boundary components. The chains and their adjoints are
exactly the toric subdivisions in these constructions.
Thus the full chains, up to reversal, agree with those in
the original basket. This preserves both their determinants
and their minimal log discrepancies.
\end{proof}

\begin{lemma}\label{r1-four}
If a rank-one $\eps$-lc Fano surface has four singular points,
then $\lcm_pD_p\le1722\eps^{-6}$.
\end{lemma}
\begin{proof}
Let $2\le g_1\le g_2\le g_3\le g_4$ be the local fundamental
group orders. By \cite[Theorem 2.1]{Belousov},
$\sum_i1/g_i\ge1$, and $D_i\mid g_i$.
If $g_1=g_2=2$, the first two points are $A_1$, so
\[
 \lcm_iD_i\le2D_3D_4\le288\eps^{-6}.
\]
Otherwise, if $\sum_{i=1}^3 1/g_i\ge1$, the first triple is
$(2,3,j)$ with $3\le j\le6$, $(2,4,4)$, or $(3,3,3)$.
Its least common multiple is at most thirty, giving
$\lcm_iD_i\le30D_4\le360\eps^{-3}$.

In the remaining case, Lemma~\ref{four-orders-arithmetic}
gives $\lcm(g_1,g_2,g_3,g_4)\le1722$.
Since $D_i\mid g_i$, the assertion follows.
\end{proof}

\begin{proposition}[A seventh-order common Weil index]\label{rank-one-index}
Let $X$ be an $\eps$-lc Fano surface with $\rho(X)=1$.
For $q=\lcm_{p\in\Sing X}D_p$, the divisor $qD$ is Cartier
for every integral Weil divisor $D$, and
\begin{equation}\label{global-index-bound}
                            q\le1728\eps^{-7}.
\end{equation}
In particular $H^2\ge\eps^7/1728$ for every integral ample $H$.
\end{proposition}
\begin{proof}
If there are at most two points, Lemma~\ref{r1-lem:local}
gives $q\le144\eps^{-6}$. With three points, if one is
noncyclic, then
$q\le(4\eps^{-1})(12\eps^{-3})^2=576\eps^{-7}$.
If all three are cyclic and one is canonical, then
$q\le(12\eps^{-1})(12\eps^{-3})^2=1728\eps^{-7}$.
If all three are noncanonical cyclic points, use
Lemma~\ref{r1-classification} and Proposition~\ref{r1-prop:toric}
on the toric realisation; it has the same determinants and
discrepancies. Four points are treated by Lemma~\ref{r1-four}.
These cases exhaust Belousov's theorem.
Finally $qH^2=(qH)\cdot H\in\Z_{>0}$.
\end{proof}

\subsection{Log Calabi--Yau polarisations}
\begin{proof}[Proof of Theorem~\ref{main-rank-volume}]
The Fano assertion is Proposition~\ref{rank-one-index}.
For a log Calabi--Yau pair, pass to the ample model of $H$
by Lemma~\ref{H-model}; this preserves integrality, $H^2$,
and the pair discrepancy bound. Now $\rho(X)=1$ and $H$
is ample. If $B\ne0$, then $-K_X\equiv B$ is ample and
Proposition~\ref{rank-one-index} applies. If $B=0$, then
Proposition~\ref{zero-absolute} gives $H^2\ge1/55440$.
Both cases imply $H^2\ge\eps^7/55440$.
Sharpness follows from Theorem~\ref{seventh-example}.
\end{proof}
\section{Sharp cubic adjoint thresholds}\label{local-sec}\label{general-threshold-sec}
\subsection{The rank-one endpoint and periodic Riemann--Roch}
The threshold problem needs additive control of local periods.
It does not require bounding their least common multiple.
\begin{proposition}[An additive-period threshold bound]\label{r1-prop:RR}
Let $X$ be a klt Fano surface with $\rho(X)=1$, and let $H$ be an integral
ample Weil divisor. Then
\begin{equation}\label{r1-eq:period-bound}
 \tau(X,H)=\lambda(X,H)\leq3+\sum_{p\in\Sing X}(D_p-1).
\end{equation}
One may replace $D_p$ by the local Cartier index of $H$ at $p$.
\end{proposition}

\begin{proof}
Write $-K_X\equiv tH$. Singular Riemann--Roch for integral Weil divisors
on a quotient surface gives, for every $m\in\ZZ$,
\begin{equation}\label{r1-eq:RR}
 F(m):=\chi(\mathcal O_X(mH))
 =1+\tfrac12(mH)\cdot(mH-K_X)+\sum_p c_p(mH).
\end{equation}
Each $c_p$ depends only on the local divisor class. Consequently
$m\mapsto c_p(mH)$ is periodic with a period dividing $D_p$.
This is the local-class form of singular Riemann--Roch; see \cite[Section 3]{HP} and \cite{Blache}.

Let $E$ be the forward shift on functions of $m$ and define
\[
 Q(z)=(1-z)^3\prod_p(1+z+\cdots+z^{D_p-1}),\qquad
 d=\deg Q=3+\sum_p(D_p-1).
\]
The factor $(1-E)^3$ annihilates the quadratic polynomial in
\eqref{r1-eq:RR}. For every $p$, the polynomial $Q$ is divisible by
$1-z^{D_p}$, so it also annihilates that periodic correction. Thus
\begin{equation}\label{r1-eq:recurrence}
 Q(E)F=0.
\end{equation}
If $t>d$, then for $1\leq j\leq d$ the divisor
$-jH-K_X\equiv(t-j)H$ is ample. Kawamata--Viehweg vanishing for integral $\Q$-Cartier Weil divisors \cite[Theorem 2.70]{KM} gives
\[
 F(-j)=h^0(X,\mathcal O_X(-jH))=0.
\]
The last equality follows from ampleness of $H$. But $F(0)=1$, since
a klt Fano surface is rational. Evaluating \eqref{r1-eq:recurrence} at
$m=-d$, all terms except the leading coefficient of $Q$ times $F(0)$
vanish. The leading coefficient is $-1$, a contradiction.
\end{proof}

Lemma~\ref{r1-lem:local} now gives
\begin{equation}\label{rank-one-threshold}
             \tau(X,H)=\lambda(X,H)\le18\eps^{-3}
\end{equation}
on every rank-one $\eps$-lc Fano surface.
\subsection{Exact formulas for one birational step}
\begin{lemma}\label{scaling-step}
Let $g:X\to Z$ contract a $K_X$-negative extremal curve $C$, and let $H$ be integral, nef and big. Set
\[
w=-C^2>0,\qquad d=H\cdot C\ge0,\qquad
u=d/w,\qquad a=(-K_X\cdot C)/w>0.
\]
Put $H_Z=g_*H$. Then $Z$ is $\eps$-lc whenever $X$ is, $H_Z$ is integral, nef and big, and
\begin{align}
H&=g^*H_Z-uC,& K_X&=g^*K_Z+aC,\label{scaling-identities}\\
H_Z^2&=H^2+d^2/w,&
K_X+tH&=g^*(K_Z+tH_Z)+(a-tu)C.\label{scaling-family}
\end{align}
If $d=0$, then $\tau(X,H)=\tau(Z,H_Z)$. If $d>0$ and $K_X+sH$ is nef with $(K_X+sH)\cdot C=0$, then $a=su$ and again the pseudo-effective thresholds are equal.
\end{lemma}
\begin{proof}
Both differences in \eqref{scaling-identities} are exceptional. Intersecting with $C$ gives their coefficients. The target is klt and $\Q$-factorial by surface contraction theory. On a common resolution the canonical identity gives
\[
a(E,Z)=a(E,X)+a\,\ord_E(C)\ge a(E,X),
\]
since an effective $\Q$-Cartier divisor has nonnegative valuation. This proves preservation of the ordinary discrepancy bound.

The pushforward $H_Z$ is an integral Weil divisor. For any curve $D_Z$ on $Z$, its strict transform $D\ne C$,
\[
H_Z\cdot D_Z=(H+uC)\cdot D\ge0.
\]
Here distinct effective curves on a normal $\Q$-factorial surface have nonnegative intersection. This implies $H_Z$ is also nef. Squaring the first identity gives \eqref{scaling-family}. Bigness of $H_Z$ follows from $H_Z^2\geq H^2>0$.

If $d=0$, the exceptional coefficient $a-tu=a$ is nonnegative for every $t\ge0$. Pushing forward a pseudo-effective divisor preserves pseudo-effectivity; pulling back a pseudo-effective divisor and adding an effective exceptional divisor does too. This proves equality of thresholds.

If $d>0$, the null-intersection condition gives $a=su$. For $0\le t\le s$ the coefficient is $(s-t)u\ge0$, so the same equivalence holds. At $t=s$ the divisor descends without exceptional error and its pushforward is nef. Therefore both thresholds are at most $s$, proving equality on their full relevant range.
\end{proof}

Equation \eqref{scaling-family} also shows that these contractions can increase $H^2$, even when they preserve the threshold. Thus a volume lower bound on the endpoint does not directly give one on the original surface.

\subsection{Reaching a Mori fibre space at the threshold}
\begin{lemma}\label{threshold-reduction}
If $H$ is integral, nef and big on a projective klt surface and $\tau(X,H)>0$, there is a sequence of $K$-negative birational contractions to a Mori fibre space $Z\to T$ such that, for a general fibre $F$,
\begin{equation}\label{endpoint-equation}
K_F+\tau(X,H)H_F\equiv0.
\end{equation}
The divisor $H_F$ is the restriction of the pushed-forward integral Weil divisor, is positive on the fibre, and every surface in the sequence preserves the original ambient discrepancy lower bound.
\end{lemma}
\begin{proof}
First contract any $H$-trivial $K$-negative extremal rays while they exist. They are birational: a fibre-type contraction has a moving general fibre on which a nef big divisor has positive intersection. Each birational step reduces the Picard number by one and preserves $H^2$ and $\tau$, by Lemma~\ref{scaling-step}. Thus after finitely many steps all $K$-negative extremal rays are $H$-positive.

On this model, the nef threshold
\[
s=\sup_{C}\frac{-K_X\cdot C}{H\cdot C}
\]
is finite, positive, rational and achieved by an extremal curve. Here the supremum ranges over the rational curves in the cone theorem, chosen with $0<-K_X\cdot C\le4$. We verify the assertion as follows: choose fixed integers $q,j$ making $qH$ and $jK_X$ Cartier on this particular one surface. Then $H\cdot C\ge1/q$, so all ratios are at most $4q$. Above any fixed positive number, only finitely many ratios are possible: the numerator belongs to $j^{-1}\Z\cap(0,4]$, and the positive denominator belongs to $q^{-1}\Z$ and is bounded above. The supremum is therefore a maximum and is rational. These auxiliary $q,j$ are used only for existence, not in the final bound.

Put $D=K_X+sH$. It is nef. If $D$ is big, the extremal ray attaining the maximum must be birational: a nef big divisor cannot be numerically trivial on a moving general fibre. Contract that ray. Lemma~\ref{scaling-step} preserves the pseudo-effective threshold, while the pushforward of $D$ is nef. Repeat, removing any new $H$-trivial $K$-negative rays if necessary. Picard number decreases at every birational step.

Eventually we reach a model on which $D$ is nef and not big. Indeed, continuing with a big $D$ would always supply another birational contraction. Necessarily $s=\tau$: if $s>\tau$, then
\[
D=(K_X+\tau H)+(s-\tau)H
\]
would be big, since the pseudo-effective cone is closed and $H$ is big. A Cartier multiple of $D$ is semiample by base-point-freeness, because $D-K_X=sH$ is nef and big. Its morphism has image a point or a curve.

Run a relative $K_X$-MMP over this image. The canonical divisor is not relatively pseudo-effective: on a general fibre $D\equiv0$ and $K_X\equiv-sH$ with $H$ positive. Every birational step is $D$-trivial, so Lemma~\ref{scaling-step} again preserves $\tau=s$. The relative MMP ends at a Mori fibre space $Z\to T$. Restricting the descended equality $D\equiv_T0$ gives \eqref{endpoint-equation}. All contractions are $K$-negative contractions, so the discrepancy assertion follows from the same lemma.
\end{proof}

\begin{proof}[Proof of Theorem~\ref{main-thresholds}]
If $K_X$ is pseudo-effective, $\tau=0$. Otherwise apply
Lemma~\ref{threshold-reduction}. When the Mori fibre base
is a curve, the general fibre is a smooth $\PP^1$ avoiding
the singular points. Thus $d=H_Z\cdot F\in\Z_{>0}$ and
$-2+\tau d=0$, so $\tau\le2$.
When the base is a point, $Z$ is a rank-one $\eps$-lc Fano
surface and the pushed-forward $H_Z$ is ample.
The exact threshold equality and \eqref{rank-one-threshold}
give $\tau(X,H)\le18\eps^{-3}$.
Theorem~\ref{seventh-example} proves sharpness.
\end{proof}

\begin{remark}\label{integer-threshold}
Taking $l=\lfloor18\eps^{-3}\rfloor+1$ gives $K_X+lH$ big:
add $(l-\tau)H$ to $K_X+\tau H$.
If $H-K_X$ is pseudo-effective, then
$lH-K_X=(l-1)H+(H-K_X)$ is also big.
\end{remark}

\begin{remark}[The nef threshold]\label{nef-adjoint}
The pseudo-effective bound does not imply a finite nef
threshold for a nef and big polarisation. On
$X=\operatorname{Bl}_p\PP^2$, let $H$ be the pullback
of a line and $E$ the exceptional curve.
Then $(K_X+tH)\cdot E=-1$ for every $t$, so
$\lambda(X,H)=+\infty$. In Picard number one with
$-K_X$ ample, the two thresholds coincide.
\end{remark}
\section{Weak Fano and big-boundary log Calabi--Yau pairs}\label{big-sec}
\begin{proposition}\label{big-boundary-thm}
Let $(X,B)$ be a projective $\eps$-lc log Calabi--Yau $\Q$-pair with $B$ big, and let $H$ be integral, nef and big. Then
\[
H^2\ge c_b\frac{\eps^9}{\ellog(\eps)}.
\]
The same bound holds when $X$ is $\eps$-lc weak Fano.
\end{proposition}
\begin{proof}

First suppose $(X,B)$ is $\eps$-lc log Calabi--Yau and $B$ is big. Such $X$ is of Fano type, hence rational. Indeed, write $B\sim_\Q A+E$ with $A$ ample and $E\ge0$, and take the effective boundary $\Delta_t=(1-t)B+tE$ for sufficiently small $t>0$. It is klt, and $K_X+\Delta_t\sim_\Q-tA$, giving a log Fano pair. For the rationality assertion, Kawamata--Viehweg vanishing \cite[Theorem 2.70]{KM} and rational singularities give irregularity zero on the resolution. The anticanonical divisor is big on the resolution, so its Kodaira dimension is $-\infty$; the vanishing of irregularity forces the ruled base to be $\PP^1$.

We may use Lemma~\ref{H-model} to assume $H$ is ample. The new boundary remains big and the pair remains $\eps$-lc. If $\rho(X)=1$, then $-K_X$ is ample. Proposition~\ref{rank-one-index} gives an integer $q\le1728\eps^{-7}$ making $H$ Cartier after multiplication by $q$. Therefore $qH^2\in\Z_{>0}$ and
\begin{equation}\label{rank-one}
H^2\ge\frac{\eps^7}{1728}.
\end{equation}

If $\rho(X)>1$, Theorem~\ref{main-thresholds} shows that $K_X+T(\eps)H$ is pseudo-effective. The anticanonical model in Lemma~\ref{anti-model} is $\eps$-lc weak Fano. Thus \eqref{fano-input} and Lemma~\ref{positive-part} give
\begin{equation}\label{big-result}
H^2\ge\frac{\vol(-K_X)}{T(\eps)^2}
\ge\frac{A_F}{324}\frac{\eps^9}{\ellog(\eps)}.
\end{equation}
Together with \eqref{rank-one}, this proves the big-boundary part of Theorem~\ref{main}.

Now for an $\eps$-lc weak Fano surface and $\eps<1$, choose a sufficiently  general anticanonical member $G$ and put $B=G/m\sim_\Q-K_X$, with $1/m\le1-\eps$. Since the anticanonical system is semiample, a general member avoids the singular points and the images of exceptional curves in a fixed resolution, so $(X,B)$ is $\eps$-lc. The big boundary argument above applies immediately. For $\eps=1$, apply it for every $\eps'<1$ and pass to the limit.

\end{proof}

\begin{corollary}\label{weak-log-variant}
The volume estimate in Theorem~\ref{main}, with a change of constant, holds for an $\eps$-lc log weak Fano pair, allowing an $\R$-boundary.
\end{corollary}
\begin{proof}
The same construction works starting with an $\eps$-lc log weak Fano $\Q$-pair $(X,\Delta)$: a small-coefficient general effective representative of $-(K_X+\Delta)$ produces a big log Calabi--Yau boundary. One may first replace $\eps$ by $\eps/2$; this only changes the absolute constant. For an $\R$-boundary $\Delta$, write $D=-(K_X+\Delta)\sim_\R A+E$ with $A$ ample and $E\ge0$. Replacing $\Delta$ by $\Delta+tE$ for small $t>0$ makes its negative log canonical class ample, since it is numerically $(1-t)D+tA$, and retains $(3\eps/4)$-lc singularities. As ampleness is an open condition, any nearby rational effective boundary retains ampleness and is $\eps/2$-lc. The rational construction therefore applies to a sufficiently small rational perturbation of $\Delta+tE$.\end{proof}

\section{Non-big boundaries and genus-one fibrations}\label{elliptic-sec}
The goal in this branch is the intersection estimate
$H\cdot P_X\ge\eps^2/36$, where $P_X$ is the positive
part of $-K_X$. The argument has three steps: at most
two fibres can be modified; each correction has reduced
denominator at most $6/\eps$ after scaling by the
Halphen index; and that index cancels against the
integral fibre degree of $H$. The cubic threshold
then gives $H^2\ge\eps^5/648$.

Let $(X,B)$ be an $\eps$-lc log Calabi--Yau pair. By Lemma~\ref{H-model}, we may assume $H$ ample. If the pushed-forward boundary has become big, Section~\ref{big-sec} applies. Otherwise apply the anticanonical MMP of Lemma~\ref{anti-model}.

If $-K_Y\equiv0$, then $B_Y=0$, and Section~\ref{zero-sec} applies. Suppose instead $-K_Y$ is nonzero, nef, and has self-intersection zero. By semiampleness it defines a morphism to a curve. A general fibre is smooth of genus one, because its canonical degree $K_Y\cdot F_Y$ is zero. Composing the maps gives a genus-one fibration on $X$ and on its minimal resolution $f:S\to X$.

The crepant boundary $B_S$ is effective, and $(S,B_S)$ is $\eps$-lc. It is vertical because it has zero intersection with a general fibre. Every $f$-exceptional curve is also vertical. Contract vertical $(-1)$-curves to reach a relatively minimal genus-one surface
\[
\pi:S\longrightarrow S_0\longrightarrow C.
\]
The pushed-forward boundary $B_0$ remains effective and klt, with $K_{S_0}+B_0\sim_\Q0$. It is a nonzero vertical boundary. Indeed, if $B_0=0$, the formula for $K_S$ under point blow-ups and effectivity of $B_S$ would force both the blow-up exceptional divisor and $B_S$ to vanish, contradicting $B_Y\neq 0$. The canonical bundle formula \cite[V, Section 12]{BHPV} gives
\begin{equation}\label{canonical-elliptic}
K_{S_0}\equiv
\left(2g(C)-2+\chi(\mathcal O_{S_0})+
\sum_i(1-1/M_i)\right)F,
\end{equation}
where $M_i$ are the multiple-fibre multiplicities and $F$ is a general fibre. Its coefficient is negative, since $B_0\ne0$ is effective and $K_{S_0}\equiv-B_0$.
Here the required inequality $\chi(\mathcal O_{S_0})\ge0$ follows from topology. The canonical bundle formula gives $K_{S_0}^2=0$, and Noether's formula gives
\[
12\chi(\mathcal O_{S_0})=e(S_0).
\]
A smooth elliptic fibre has topological Euler number zero. By the additivity of Euler number, together with local triviality over the complement of the singular fibres, we have
\[
e(S_0)=\sum_{t\in C}e((F_t)_{\mathrm{red}})\ge0.
\]
The inequality follows from the Kodaira classification: each singular reduced fibre has positive Euler number, as recalled below; a smooth multiple fibre contributes zero because multiplicity does not change the underlying topological space. Thus \eqref{canonical-elliptic} forces $C=\PP^1$ and $\chi(\mathcal O_{S_0})\in\{0,1\}$.

If $\chi=0$, the Euler number is zero, so all reduced fibres are smooth elliptic curves. The vertical boundary on $S_0$ has disjoint smooth support with coefficients less than one. Blowing up any point would give a negative coefficient in its crepant transform. Since $B_S$ is effective, no such blow-up occurs: $S=S_0$. There are no vertical rational curves to contract by $f$, so $X$ itself is smooth and we have $H^2\ge1$. We may therefore assume $\chi=1$.

If $\chi=1$, then $S_0$ is rational: in this case the base is $\PP^1$, its canonical class has negative fibre coefficient, so $p_g(S_0)=0$; the identity $1=\chi=1-q+p_g$ then gives $q(S_0)=0$, so the smooth surface classification applies. The inequality $\sum(1-1/M_i)<1$ permits at most one multiple fibre. If there is no multiple fibre, set $m=1$; otherwise let $m$ be its multiplicity. Formula \eqref{canonical-elliptic} gives $K_{S_0}\equiv-F/m$. Numerical and linear equivalence agree for integral divisors on a smooth rational surface, so it is a Halphen surface of index $m$ and
\begin{equation}\label{halphen}
F\sim-mK_{S_0}.
\end{equation}
There is no multiple fibre for $m=1$, and a unique multiple fibre of multiplicity $m$ otherwise. The total Euler number of its singular fibres is $12\chi(\mathcal O_{S_0})=12$. The integer $m$ need not be bounded; it will cancel in a later intersection argument. For background on Halphen surfaces, see Cantat--Dolgachev \cite[Proposition 2.2]{Halphen}. For the fibre configurations and Euler numbers see Sch\"utt--Shioda \cite[Sections 4 and 6]{Elliptic}; the canonical bundle formula with multiple fibres is also recalled in \cite{BHPV}. 

\label{fibre-sec}
\subsection{The explicit Zariski decomposition}
Assume $S_0$ is rational. Write
\[
K_S=\pi^*K_{S_0}+\sum_E k_EE,\qquad k_E\in\Z_{>0},
\]
over the $\pi$-exceptional curves. If $E$ lies over a fibre $F_i$, let $n_E$ be its multiplicity in the total pullback of that fibre. Define
\begin{equation}\label{gamma}
\gamma_i=\max_{E\text{ over }F_i}\frac{k_E}{n_E},
\end{equation}
with value zero when there is no exceptional curve there.

On the relatively minimal surface, the vertical boundary is a sum of whole fibres:
\[
B_0=\sum_i t_iF_i,\qquad t_i\ge0,\qquad\sum_i t_i=1/m.
\]
Indeed, $B_0$ is numerically orthogonal to every component of every fibre. The one-dimensional kernel of each connected fibre matrix therefore forces proportionality to its multiplicity vector. Effectivity of $B_S$ gives $t_i\ge\gamma_i$.

Put $d=1/m-\sum_i\gamma_i\ge0$. Then
\begin{equation}\label{elliptic-zariski}
-K_S\sim_\Q dF+\sum_i\left(\gamma_i\pi^*F_i-\sum_{E\text{ over }F_i}k_EE\right).
\end{equation}
Write
\[
 N_i=\gamma_i\pi^*F_i-\sum_{E\text{ over }F_i}k_EE.
\]
Every nonzero $N_i$ is effective, and its support is a
proper subset of the corresponding fibre: an exceptional
component attaining the maximum in \eqref{gamma} has
coefficient zero. The divisor $dF$ is nef and orthogonal
to every $N_i$. A proper subset of a connected fibre has
negative definite intersection matrix, and different
fibres are disjoint. Hence $\sum_iN_i$ is negative
definite on its support, proving that
\eqref{elliptic-zariski} is the Zariski decomposition.
In the branch under consideration, $d>0$.

The positive part is also $f^*P_X$, where $P_X$ is the positive part of $-K_X$. Indeed, adding the effective $f$-exceptional discrepancy divisor to $f^*(-K_X)$ leaves the positive part unchanged. The negative support remains negative definite by the orthogonal decomposition into pullback and exceptional classes.

\subsection{The Kodaira numbers used in the argument}
We explain two different multiplicity bounds. The largest \emph{component} multiplicity on a fixed normal-crossing resolution is 6. The largest \emph{point} multiplicity on the relatively minimal surface can be 11. The first controls denominators of later valuations; the second contributes to the first blow-up of a modified fibre. 

For a fibre $F_i=M_i\overline F_i$, call $\overline F_i$ the primitive fibre divisor. Its component multiplicities are the relatively prime positive integers in the kernel of its intersection matrix. At a transverse intersection of components of multiplicities $u,v$, its multiplicity at the point is $u+v$. A first blow-up there consequently has fibre multiplicity $M_i(u+v)$.

For the starred fibres all components are smooth rational $(-2)$-curves and the reduced support has ordinary double crossings. Their primitive multiplicities are described as follows. The notation for an arm lists its multiplicities starting next to the central component.

\begin{table}[htbp]\centering\small
\caption{Primitive component and point multiplicities}\label{tab:kodaira-components}
\begin{tabular}{@{}>{\raggedright\arraybackslash}p{.10\textwidth}>{\centering\arraybackslash}p{.10\textwidth}>{\raggedright\arraybackslash}p{.36\textwidth}>{\centering\arraybackslash}p{.13\textwidth}>{\centering\arraybackslash}p{.18\textwidth}@{}}
\toprule
Type & Central & Arms or chain & Components & Maximum at a point\\
\midrule
$I_0^*$ & $2$ & Four arms $(1)$ & $5$ & $3$\\
$I_n^*$, $n\ge1$ & --- & $n+1$ vertices of weight $2$; two $1$-tips at each end & $n+5$ & $4$\\
$IV^*$ & $3$ & $(2,1),(2,1),(2,1)$ & $7$ & $5$\\
$III^*$ & $4$ & $(2),(3,2,1),(3,2,1)$ & $8$ & $7$\\
$II^*$ & $6$ & $(3),(4,2),(5,4,3,2,1)$ & $9$ & $11$\\
\bottomrule
\end{tabular}
\end{table}

To check these multiplicities, at each vertex $v$, the fibre equation is $2m_v=\sum_{w\sim v}m_w$. The displayed vectors solve these equations and have greatest common divisor one. Since the fibre matrix has a one-dimensional kernel, they are the primitive multiplicity vectors. The dual graph is a tree, so its Euler number is one more than its number of components: this gives $n+6,8,9,10$ for $I_n^*,IV^*,III^*,II^*$ respectively. In particular the number eleven is $6+5$ at the corresponding node of $II^*$, not a component multiplicity.

For $I_n$, all component multiplicities are one, and the Euler number is $n$. The node multiplicity is two, including the two local branches of $I_1$. For the remaining types, II is a cusp, III consists of two tangent components, and IV consists of three components through one ordinary triple point. The largest point multiplicities are $2,2,3$, and the Euler numbers are $2,3,4$.

We check the bound $6$ after resolving the support. All equations below concern the primitive fibre; multiply the stated numbers by $M_i$ for a multiple fibre.
\begin{enumerate}[label=(\roman*)]
\item For a cusp $y^2-x^3=0$, the substitution $x=u,y=uv$ gives $u^2(v^2-u)$. The first exceptional component has multiplicity two and is tangent to the strict transform. At the next blow-up, $u=vw$ gives $v^3w^2(v-w)$: the new multiplicity is three and there is a triple point with branch multiplicities $3,2,1$. Blowing up that point gives multiplicity $6$ and a simple normal-crossing support.
\item For the tangency $y(y-x^2)=0$, the substitution $y=xv$ gives $x^2v(v-x)$. Blowing up the remaining triple point, with branch multiplicities $2,1,1$, produces multiplicity four.
\item Blowing up an ordinary triple point of a type IV fibre produces multiplicity three. Blowing up a node of an irreducible $I_1$ fibre produces multiplicity two and separates its branches.
\end{enumerate}
No further resolution is needed for the other Kodaira supports. Thus every component on this fixed resolution has multiplicity at most $6M_i$. For the Kodaira classification, see \cite{Kodaira}, \cite[Sections 4 and 6]{Elliptic} and \cite{BHPV}.

\subsection{A local multiplicity bound}
\begin{lemma}\label{multiplicity}
Let $M_i$ be the multiplicity of a fibre on $S_0$; it is one, except possibly for the unique multiple fibre, where it is $m$. Every exceptional curve $E$ on $S$ over that fibre satisfies
\begin{equation}\label{multiplicity-bound}
n_E\le\frac{6M_i}{\eps}.
\end{equation}
Also, no blow-up in $S\to S_0$ occurs over a smooth reduced fibre.
\end{lemma}
\begin{proof}
We may resolve the support of a Kodaira fibre on $S_0$ to simple normal crossings. The previous explicit resolutions show that every component of the total fibre on this fixed model has multiplicity at most $6M_i$.

On this fixed log resolution the pair discrepancies are positive and at least $\eps$, although some boundary coefficients may be negative. A further valuation with log discrepancy at most one can only arise by subdivision of a node. A blow-up at a point on only one boundary component has discrepancy $1+\alpha>1$, and subsequent blow-ups involving that component cannot return to discrepancy at most one. A node valuation has relatively prime positive coordinates $a,b$ and
\[
\alpha_E=a\alpha_1+b\alpha_2,\qquad n_E=an_1+bn_2.
\]
For a component on $S$, effectivity of $B_S$ gives $\alpha_E\le1$. Hence $a+b\le1/\eps$, which proves \eqref{multiplicity-bound}. Components already on the fixed log resolution satisfy it directly.

Over a smooth reduced fibre the boundary is a single smooth curve of coefficient less than one. The first blow-up has discrepancy greater than one and thus cannot occur in an effective crepant model. This proves the last claim.
\end{proof}

\subsection{Why there are at most two nonzero corrections}
We call a fibre modified if $\gamma_i>0$. Let $e_i$ be the Euler number of its reduced Kodaira fibre and let $\mu_i$ be an upper bound for the largest multiplicity at a point of the primitive fibre divisor on $S_0$. The following bounds follow directly from the Kodaira configurations:

\begin{table}[htbp]\centering\small
\caption{Euler numbers and bounds for point multiplicities}\label{tab:kodaira-euler}
\begin{tabular}{c|ccccccc}
\toprule
Type & $I_n$ $(n\ge1)$ & II & III & IV & $I_n^*$ & $IV^*$ & $III^*,II^*$\\
\midrule
$e_i$ & $n$ & $2$ & $3$ & $4$ & $n+6$ & $8$ & $9,10$\\
$\mu_i$ & $2$ & $2$ & $2$ & $3$ & $4$ & $5$ & $7,11$\\
\bottomrule
\end{tabular}\end{table}

For $I_0^*$ we use the uniform bound four, although the maximum is three.

The first blow-up over a modified fibre has $k_E=1$ and $n_E\le M_i\mu_i$. Therefore
\begin{equation}\label{first-correction}
m\gamma_i\ge\frac{m}{M_i\mu_i}\ge\frac1{\mu_i}.
\end{equation}
On the other hand, $d>0$ gives $\sum_i m\gamma_i<1$.

\begin{lemma}\label{two-fibres}
At most two fibres are modified.
\end{lemma}
\begin{proof}
Suppose at least three fibres are modified. Their Euler numbers sum to at most $12$. If one has Euler number at least eight, the other two have total Euler number at most four; each of those two has $\mu\le2$, so their reciprocal contributions already sum to at least $1$.

Otherwise, if one is of type $I_n^*$, its Euler number is at least $6$ and its reciprocal contribution is at least $1/4$. The other two have total Euler number at most $6$. Either both have $\mu\le2$, or one is type IV and the other has $\mu\le2$. In the latter, smaller case the total contribution is at least $1/4+1/3+1/2=13/12>1$.

In all remaining cases every fibre has $\mu\le3$, so three reciprocal contributions sum to at least $1$. Each possibility contradicts \eqref{first-correction} and $\sum m\gamma_i<1$.
\end{proof}

\subsection{The lower bound for volume in this case}
\begin{proposition}\label{elliptic-degree}
In the rational elliptic branch,
\begin{equation}\label{HP}
H\cdot P_X\ge\frac{\eps^2}{36}.
\end{equation}
In particular, if $H$ is ample on $X$, then
\begin{equation}\label{elliptic-volume}
H^2\ge\frac{\eps^5}{648}.
\end{equation}
\end{proposition}
\begin{proof}
We can choose a component attaining each nonzero maximum in \eqref{gamma}. Lemma~\ref{multiplicity} shows that $m\gamma_i$ is a rational number whose reduced denominator is at most $6/\eps$. For a non-multiple fibre its denominator divides $n_E\le6/\eps$. For the multiple fibre, $n_E=m\bar n_E$, so $m\gamma_i=k_E/\bar n_E$ with $\bar n_E\le6/\eps$.

By Lemma~\ref{two-fibres}, the positive number
\[
md=1-\sum_i m\gamma_i>0
\]
has denominator at most $(6/\eps)^2$. Thus we have $md\ge\eps^2/36$.

Let $\widetilde H$ be the integral strict transform on $S$. All $f$-exceptional curves are vertical. Using \eqref{halphen},
\[
H\cdot F_X=f^*H\cdot F=\widetilde H\cdot F
=-m\,\widetilde H\cdot\pi^*K_{S_0}\in m\Z_{>0}.
\]
The positivity follows from nefness and bigness of $H$ and that $F_X$ is a fibre. Therefore
\[
H\cdot P_X=d\,H\cdot F_X\ge md\ge\eps^2/36.
\]
For the volume estimate, Theorem~\ref{main-thresholds}
gives $K_X+18\eps^{-3}H$ pseudo-effective. Since $H$
is nef and $-K_X=P_X+N_X$ with $N_X\ge0$,
\[
 18\eps^{-3}H^2\ge H\cdot(-K_X)
                    \ge H\cdot P_X\ge\frac{\eps^2}{36}.
\]
This proves \eqref{elliptic-volume}.
\end{proof}

\section{Numerically trivial anticanonical models}\label{zero-sec}
We first treat $K_X\equiv0$ itself. This gives an absolute
volume bound without a restriction on the Picard number.
The case where only the anticanonical nef model has
numerically trivial canonical class requires a separate argument.

\begin{proposition}\label{zero-absolute}
Let $X$ be a projective klt surface with $K_X\equiv0$.
Every integral nef and big Weil divisor $H$ satisfies
\begin{equation}\label{CY-absolute}
                   H^2\ge\frac1{66\cdot840}=\frac1{55440}.
\end{equation}
\end{proposition}
\begin{proof}
By surface log abundance \cite[Theorem 1.1]{Fujino},
$K_X\sim_\Q0$. Let $I$ be the least positive integer with
$IK_X\sim0$. The associated connected cyclic canonical cover
$\pi:\widehat X\to X$ is quasi-\'etale of degree $I$ and
$K_{\widehat X}\sim0$; see \cite[Section 5.2]{KM}.
Its singularities are Gorenstein and klt, hence canonical.
The minimal resolution $\widetilde X$ has trivial canonical
bundle, so it is a K3 or abelian surface
\cite[Chapter VI]{BHPV}.

The deck action lifts to $\widetilde X$ and acts on its
holomorphic two-form by a primitive $I$-th root of unity.
Otherwise a smaller positive multiple of $K_X$ would be
trivial. The corresponding cyclotomic polynomial divides
the characteristic polynomial on $H^2(\widetilde X,\Q)$.
Thus $\varphi(I)\le22$ in the K3 case, and $\varphi(I)\le6$
in the abelian case. In particular $I\le66$. Indeed,
\[
 \varphi(I)=\prod_{p^a\Vert I}p^{a-1}(p-1)
\]
restricts the possible primes to $2,3,5,7,11,13,17,19,23$;
the largest integer with totient at most $22$ is $66$.

The pullback $\widehat H=\pi^*H$ is integral because $\pi$
is quasi-\'etale, and $\widehat H^2=IH^2$.
In the K3 case, the exceptional ADE configurations have
total rank at most nineteen: their classes lie in the
orthogonal complement of the positive pullback of
$\widehat H$ in a N\'eron--Severi lattice of rank at most
twenty. An abelian surface contains no exceptional rational
curves. Let $q$ be the least common multiple of the exponents
of the local class groups of $\widehat X$.
Lemma~\ref{ade-budget} gives $q\le840$. Since $q\widehat H$
is Cartier,
\[
              1\le(q\widehat H)\cdot\widehat H=qIH^2,
\]
which proves the proposition.
\end{proof}

\begin{proposition}\label{zero-thm}
Let $(X,B)$ be a projective klt log Calabi--Yau $\Q$-pair
whose anticanonical nef model has numerically trivial
canonical class. Every integral nef and big Weil divisor
$H$ satisfies $H^2\ge c_0$.
\end{proposition}
\begin{proof}
Let $g:X\to Y$ be the anticanonical nef model. Its
pushed-forward boundary is effective and numerically
trivial, hence zero. Thus
\[
                         K_X+B=g^*K_Y.
\]
The canonical-cover argument in Proposition~\ref{zero-absolute}
gives an integer $I\le66$ with $IK_Y\sim0$.
Every log discrepancy over $Y$ is a positive integral
multiple of $1/I$. Consequently $(X,B)$ is $1/66$-lc,
independently of the original discrepancy bound.

If $X$ is rational, Proposition~\ref{mass} and
\eqref{potential}, applied with $\eps=1/66$, bound every
local determinant by $344\cdot66^3$ and the number of
singular points by $18\cdot66^3$. Their product clears
every local Weil class. Therefore
\begin{equation}\label{zero-volume}
             H^2\ge(344\cdot66^3)^{-18\cdot66^3}=c_0.
\end{equation}

Suppose now that $X$ is not rational. On the minimal
resolution $S$, the effective crepant boundary $B_S$ is
supported on rational curves exceptional over $Y$.
If $\kappa(S)=-\infty$, the classification of smooth
projective surfaces gives a ruled minimal model, whose
ruling pulls back to a morphism $S\to C$.
The base has positive genus, since otherwise $X$ would be
rational. Every rational curve on $S$ is then vertical.
For a general ruling fibre $F$ this gives $B_S\cdot F=0$
and $K_S\cdot F=-2$, contradicting $K_S+B_S\sim_\Q0$.
Hence $\kappa(S)\ge0$. A positive multiple of $K_S$ is
effective, and its sum with the corresponding multiple
of $B_S$ is linearly trivial. Intersecting with an ample
divisor forces both effective divisors to vanish.
Thus $B_S=0$, $K_S\equiv0$, and $X$ has canonical
singularities with $K_X\equiv0$.
Proposition~\ref{zero-absolute} gives
$H^2\ge1/55440\ge c_0$.
\end{proof}

\subsection{Proof of the main theorem}\label{completion-sec}
\begin{proof}[Proof of Theorem~\ref{main}]
Apply the ample-model reduction of Lemma~\ref{H-model}.
If the resulting boundary is big, use Section~\ref{big-sec}.
Otherwise run the anticanonical MMP. Its semiample
anticanonical divisor is either numerically trivial,
treated above, or defines the genus-one fibration of
Section~\ref{elliptic-sec}. In the latter case a nonrational
surface is smooth, while a rational surface satisfies
\eqref{elliptic-volume}.

Since $\eps^5\ge\eps^9/\ellog(\eps)$ and
$1\ge\eps^9/\ellog(\eps)$, the constants in
\eqref{more-constants} prove \eqref{main-cy}.
The inequality $\ellog(\eps)\le2/\eps$ gives the pure
tenth-order polynomial bound.

For an $\R$-boundary with $K_X+B\equiv0$, fix a log
resolution and its finite support. The numerical equations
for the coefficients have rational coefficients.
Approximate within their solution space by a rational
effective boundary $B'$ sufficiently close to retain
$\eps/2$-lc singularities. This can be checked on the fixed
log resolution, since further blow-ups of a simple normal
crossing pair cannot decrease the discrepancy below the
minimum of the positive discrepancies already present.
Surface log abundance gives $K_X+B'\sim_\Q0$.
Applying the rational-boundary theorem changes the constant
by at most a factor $2^{10}$ and leaves the order unchanged.
\end{proof}
\section{Polynomial effective birationality}\label{birational-sec}
We refine the surface alternative in \cite[Lemma 4.1]{Bie}.
The centre construction and cutting lemma are those of
\cite[Sections 2.15 and 3.1, Lemma 2.18]{Birkar}.
The quantitative input is Lemma~\ref{centre-degree}: the
degree of a general curve centre is bounded below linearly
in $\eps$. This yields the conversion stated in
Theorem~\ref{intro-conversion}. The argument uses ambient
$\eps$-lc singularities and the adjoint positivity
$H-K_X$; it does not use a log Calabi--Yau boundary or
the threshold theorem.

\subsection{Potential birationality and additions}
A big $\Q$-Cartier $\Q$-divisor $D$ on $X$ is said to be \emph{potentially birational} if, for two very general points $x,y\in X$, after possibly interchanging them, there exist a rational $0<\delta<1$ and an effective $\Q$-divisor
\[
\Delta\sim_\Q(1-\delta)D
\]
such that $(X,\Delta)$ is lc near $x$, the point $x$ is an isolated non-klt centre, and $(X,\Delta)$ is not klt at $y$. The standard implication is the following
\begin{equation}\label{pb-implication}
D\text{ potentially birational}\quad\Longrightarrow\quad |K_X+\lceil D\rceil|\text{ birational};
\end{equation}
see \cite[Lemma 2.3.4]{HMXBir} and \cite[Section 2.12]{Birkar}. As all applications below are integral divisors, no rounding error enters the bounds.

\begin{lemma}\label{pb-addition}
If $D$ is potentially birational and $E$ is a pseudo-effective $\Q$-Cartier $\Q$-divisor, then $D+E$ is potentially birational. More generally, an effective divisor $\Delta\sim_\Q A$ with the required point-centre properties can be extended to a witness for $D$ whenever $D-A$ is big.
\end{lemma}
\begin{proof}
For the first assertion, let $\Delta\sim_\Q(1-\delta)D$ be a witness. Choose a smaller rational $\delta'>0$. The divisor
\[
(1-\delta')(D+E)-(1-\delta)D
=(\delta-\delta')D+(1-\delta')E
\]
is big. Choose an effective $\Q$-representative avoiding $x$ and $y$. Adding it to $\Delta$ preserves the local centre properties. Thus $D+E$ is potentially birational. For the second assertion, bigness is an open condition, so $(1-\delta')D-A$ is big for sufficiently small positive $\delta'$, and the same argument applies. The choice of very general $x,y$ ensures that they avoid the relevant stable base loci.
\end{proof}

\subsection{An intersection gap for moving curves}
\begin{lemma}\label{moving-gap-bir}
Let $G$ be a general integral member of an irreducible covering family of curves on a normal projective surface $X$. Suppose $(X,G)$ is plt near $G$ and $(G,\Diff_G(0))$ is $\eta$-lc. Then
\[
G^2=0\quad\text{or}\quad G^2\ge\eta.
\]
If $G^2=0$, a general member avoids $\Sing X$.
\end{lemma}
\begin{proof}
Since the covering family is irreducible, after restricting the parameter space to a dense open subset, any two general members occur as fibres of the same algebraic family over a connected curve. Hence they are algebraically equivalent and thus have the same numerical class as well. We claim that this class is nef: indeed, for any fixed curve, choose a distinct member and use the non-negativity of intersections of distinct effective curves. Take two distinct general members $G,G'$. At an intersection point $p$, the classification of plt surface pairs identifies a cyclic quotient chart with local class-group order $r_p$ and different coefficient $1-1/r_p$; see \cite[Chapter 4]{KM}. The coefficient bound gives $r_p\le\eta^{-1}$. The divisor $r_pG'$ is Cartier near $p$, and its restriction to the normal curve $G$ has positive integral multiplicity. Therefore
\[
(G\cdot G')_p\ge\frac1{r_p}\ge\eta.
\]
If the members meet, adding the local intersections proves the gap. If they do not meet, their common intersection number is zero. In the latter case, a singular point lying on general members would belong to two of them, a contradiction. Since $\Sing X$ is finite, a general member avoids it.
\end{proof}

\subsection{Adjunction on a curve centre}
We explain the part of covering-family adjunction that removes the factorial dependence. Suppose $X$ is $\eps$-lc, $N$ is integral, nef and big, and $n$ is the least positive integer with
\begin{equation}\label{centre-volume}
n^2N^2>16.
\end{equation}
Birkar's general construction in \cite[Section 2.15(2)]{Birkar} supplies a bounded family of centres and, for very general $x,y$, an effective divisor
\[
\Delta_0\sim_\Q nN
\]
such that $(X,\Delta_0)$ is lc near $x$, has a unique non-klt place whose centre contains $x$, and is not klt at $y$. Write $G$ for this centre. It has dimension zero or one. The constant $16=(2d)^d$ is exactly the volume threshold in that construction for $d=2$; it provides the required multiplicity positivity. No sharpness is asserted for this dimensional constant.

If $G$ is a point, it equals $\{x\}$ and already has the required isolated-centre property. Such centres require only the extension in Lemma~\ref{pb-addition}; no adjunction or cutting is needed. It therefore suffices to treat a curve centre. The bounded family has finitely many irreducible components; discard those not covering $X$ and replace the others by dense open subsets. Thus $G$ can be taken general in an irreducible covering family. Its coefficient in $\Delta_0$ is one.

Let $F\to G$ be the normalisation. Covering-family adjunction gives
\begin{align}
K_F+\Theta_F+P_F&\sim_\R(K_X+\Delta_0)|_F,\label{bir-adjunction}\\
K_F+\Lambda_F&=K_X|_F,\qquad \Lambda_F\le\Theta_F,
\qquad (F,\Lambda_F)\text{ sub-}\eps\text{-lc},\label{lambda-adjunction}
\end{align}
where $P_F$ is pseudo-effective. The construction and positivity are \cite[Section 3.1]{Birkar}, originating in \cite[Theorem 4.2]{HMXACC} and \cite[Construction 3.9 and Theorem 3.10]{BirkarAnti}; the comparison with the ambient discrepancy bound is \cite[Lemma 3.12]{BirkarAnti}. Since $F$ is a smooth curve, sub-$\eps$-lc simply means that every coefficient of $\Lambda_F$ is at most $1-\eps$, although some coefficients may be negative.

\begin{lemma}\label{centre-degree}
Every general curve centre above satisfies
\begin{equation}\label{curve-degree-bound}
nN\cdot G\ge\eps/2.
\end{equation}
\end{lemma}
\begin{proof}
Subtract \eqref{lambda-adjunction} from \eqref{bir-adjunction} and take degrees:
\begin{equation}\label{degree-cancellation}
s_0:=nN\cdot G=\deg(\Theta_F-\Lambda_F)+\deg P_F.
\end{equation}
Both terms on the right are non-negative. Suppose $s_0<\eps/2$. Then for every point $p\in F$,
\[
\operatorname{coeff}_p\Theta_F
\le\operatorname{coeff}_p\Lambda_F+\deg(\Theta_F-\Lambda_F)
<1-\eps/2.
\]
Thus $(F,\Theta_F)$ is klt, with discrepancy gap at least $\eps/2$.

We now identify $\Theta_F$ with the ordinary different before applying Lemma~\ref{moving-gap-bir}. In the construction of \cite[Section 3.1]{Birkar}, let $\pi:W\to X$ be the resulting dlt model and let $S$ be the strict transform of $G$. Since $G$ is a divisor on the surface, $S$ is not exceptional. The map $h:S\to F$ is a birational contraction of normal projective curves, hence actually an isomorphism. With ambient boundary zero, the auxiliary adjunction boundary is
\[
\Sigma_W=S+\sum_{E\text{ prime exceptional for }\pi}E,
\qquad (K_W+\Sigma_W)|_S=K_S+\Sigma_S.
\]
The pair $(W,\Sigma_W)$ is dlt near $S$: its boundary is bounded above by the dlt boundary in the construction. Since $h$ is an isomorphism, the discriminant definition of $\Theta_F$ identifies it with $\Sigma_S$. If an exceptional component met $S$, adjunction for a dlt surface pair would give a coefficient one point in $\Sigma_S$. This contradicts the klt property of $\Theta_F$. Hence
\[
S\cap\Exc(\pi)=\varnothing.
\]
Connectedness of the fibres of the proper birational morphism $\pi$ now shows that $\pi$ is an isomorphism over a neighbourhood of $G$. Indeed, an exceptional fibre over a point of $G$ would be connected and would contain the point on $S$, so it would meet $S$. In particular, $G$ is normal,  $(X,G)$ is plt near $G$, and we have
\begin{equation}\label{ordinary-different}
\Theta_F=\Diff_G(0).
\end{equation}
Here the direct argument uses the fact that $S$ and $F$ are curves. It requires only $\deg P_F\ge0$, so neither bigness of $P_F$ nor the denominator conclusion in \cite[Proposition 3.7]{Birkar} is required.

The moving class $G$ is nef, and $\Delta_0-G\ge0$. Therefore
\[
0\le G^2\le\Delta_0\cdot G=s_0<\eps/2.
\]
Lemma~\ref{moving-gap-bir} forces $G^2=0$ and shows that general $G$ avoids $\Sing X$. Integrality gives $N\cdot G\in\Z$, and nefness and bigness of $N$ give $N\cdot G>0$ by the Hodge index theorem. Hence $s_0=nN\cdot G\ge n\ge1$, contradicting $s_0<\eps/2$. This proves \eqref{curve-degree-bound}.
\end{proof}

In the surface case, the canonical divisor cancels in \eqref{degree-cancellation} before any extra divisor is added. This is the reason that $N+K_X$ does not enter our argument here.

\subsection{The polynomial alternative argument} The following Proposition~\ref{bir-alternative} holds under the more general assumptions that $X$ itself is $\eps$-lc and $N-K_X$ is big. No log Calabi--Yau assumption is required here.

\begin{proposition}\label{bir-alternative}
Let $X$ be an $\eps$-lc surface, and let $N$ be an integral nef and big Weil divisor with $N-K_X$ big. Define $n$ by \eqref{centre-volume}, and put
\begin{equation}\label{bir-U}
A=\lceil16/\eps\rceil,\qquad q=3An+1,
\qquad U=3A+1\le52\eps^{-1}.
\end{equation}
Both $qN$ and $qN-K_X$ are potentially birational. If $b$ is the first positive integer for which $|bN|$ is birational, then
\begin{equation}\label{polynomial-alternative}
b\le U\quad\text{or}\quad b^2N^2\le64U^2\le173056\eps^{-2}.
\end{equation}
\end{proposition}
\begin{proof}
Apply the previous centre construction. For a curve centre, Lemma~\ref{centre-degree} gives
\[
(AnN)\cdot G\ge A\eps/2\ge8>4.
\]
Here $4=d^d$ is the restricted-volume threshold in \cite[Lemma 2.18]{Birkar} for $d=2$. On a curve, the restricted volume of a positive nef divisor is its degree. 

We can apply the cutting lemma with the auxiliary nef and big divisor $AnN$. After possibly interchanging $x$ and $y$, it gives an effective divisor in class $(n+2An)N$ with a unique point centre through $x$ and non-klt behaviour at $y$. Add a general effective representative of $(A-1)nN$ avoiding the two points. The resulting divisor satisfies
\[
\Delta_1\sim_\Q3AnN
\]
and has the same point-centre properties. If the initial centre was a point, obtain this same class by adding an effective representative of $(3A-1)nN$ to $\Delta_0$.

For the two divisors, the respective residual classes are
\[
qN-3AnN=N,\qquad (qN-K_X)-3AnN=N-K_X.
\]
Both are big. Lemma~\ref{pb-addition} supplies the strict margin in potential birationality for both targets. By \eqref{pb-implication}, $|qN|$ is birational, so $b\le q\le Un$. If $n=1$, this gives the first alternative. Otherwise minimality of $n$ gives $(n-1)^2N^2\le16$, whence $n^2N^2\le 4\cdot(n-1)^2N^2\leq 64$ and the second alternative follows.
\end{proof}

\begin{theorem}[Volume-to-birationality conversion]\label{volume-to-bir}
Let $X$ be an $\eps$-lc surface and $H$ an integral nef and big Weil divisor. If $H-K_X$ is big, there is an integer $M$ satisfying
\begin{equation}\label{M-precise}
M\le52\eps^{-1}\max\{1,8/\sqrt{H^2}\}
\end{equation}
such that both $|mH+L|$ and $|K_X+mH+L|$ are birational for every integer $m\ge M$ and every integral pseudo-effective Weil divisor $L$. If $H-K_X$ is only pseudo-effective, the same assertion holds with
\begin{equation}\label{M-peff}
M\le52\eps^{-1}\max\{2,8/\sqrt{H^2}\}.
\end{equation}
\end{theorem}
\begin{proof}
In the big case take $N=H$ and $M=q$ from Proposition~\ref{bir-alternative}. For every $m\ge q$, the divisor $(m-q)H+L$ is pseudo-effective. Apply Lemma~\ref{pb-addition} to both potentially birational divisors $qH$ and $qH-K_X$, and then use \eqref{pb-implication}. The target divisors are integral, which gives precisely the two stated systems.

If $n=1$, then $q\le U$. If $n\ge2$, minimality gives $(n-1)\sqrt{H^2}\le4$, and therefore $n\le8/\sqrt{H^2}$. This proves \eqref{M-precise}.

In the pseudo-effective case use $N=2H$, since $2H-K_X=H+(H-K_X)$ is big, and put $M=2q$. Both $2qH$ and $2qH-K_X$ are potentially birational, so the same addition argument applies to every integer $m\ge2q$, including odd integers. Since $N^2=4H^2$, multiplying the bound for $q$ by two gives \eqref{M-peff}.
\end{proof}

\begin{proof}[Proof of Corollary~\ref{main-birational}]
For a log Calabi--Yau pair, $H-K_X\sim_\Q H+B$ is big. Insert Theorem~\ref{main} into \eqref{M-precise}. Since $c\le1$ and $0<\eps\le1$, one may take
\begin{equation}\label{refined-M}
M_\eps=\left\lceil416c^{-1/2}\eps^{-11/2}\sqrt{\ellog(\eps)}\right\rceil.
\end{equation}
The inequality $\ellog(\eps)\le2/\eps$ gives the pure sixth-order bound. In the variants, the volume constant changes absolutely and $-K_X$ is still pseudo-effective, so the same conversion applies.
\end{proof}

\begin{corollary}\label{rank-one-bir}
Under the log Calabi--Yau hypotheses of
Theorem~\ref{main-rank-volume}, both $|mH+L|$ and
$|K_X+mH+L|$ are birational for every integral
pseudo-effective Weil divisor $L$ and every integer
\[
             m\ge\left\lceil416\sqrt{55440}\,
                                      \eps^{-9/2}\right\rceil.
\]
If the boundary on the rank-one ample model is nonzero,
one may replace $55440$ by $1728$.
\end{corollary}
\begin{proof}
Theorem~\ref{volume-to-bir} gives a sufficient bound
$52\eps^{-1}\max\{1,8/\sqrt{H^2}\}$, since
$H-K_X\equiv H+B$ is big. Substitute
Theorem~\ref{main-rank-volume}.
\end{proof}

\begin{remark}
Several special cases give stronger bounds. If $X$ is weak Fano and $H=-K_X$, \cite[Theorem 1.1]{Fano} gives $M=O(\eps^{-5/2}\sqrt{\ellog(\eps)})$. On the rational genus one ample model, Proposition~\ref{elliptic-degree} instead gives $M=O(\eps^{-7/2})$. In the numerically trivial anticanonical branch, Proposition~\ref{zero-thm} gives $M=O(\eps^{-1})$. These more precise statements concern the model on which the corresponding volume estimate is established; \eqref{refined-M} is uniform for all the nef and big polarisations.
\end{remark}

\begin{proposition}[Polynomial equivalence]\label{bir-volume-converse}
If $H$ is integral, nef and big on a klt surface and $|mH|$ is birational, then $m^2H^2\ge1$. Consequently, among $\eps$-lc surfaces with $H-K_X$ pseudo-effective, a uniform polynomial volume gap exists if and only if a uniform polynomial effective birationality bound exists.
\end{proposition}
\begin{proof}
Resolve the rational map and its base ideal by $f:W\to X$. Write
\[
f^*(mH)=P+E,
\]
where $P$ is the base-point-free integral moving divisor and $E\ge0$ is a $\Q$-divisor. The map defined by $P$ is birational, so $P^2$ is the positive integral degree of its image and is at least one. Both $f^*(mH)$ and $P$ are nef, so
\[
m^2H^2=f^*(mH)\cdot(P+E)
\ge f^*(mH)\cdot P=P^2+E\cdot P\ge1.
\]
A bound $m\le C\eps^{-b}$ therefore gives $H^2\ge C^{-2}\eps^{2b}$, with an absolute adjustment for integer rounding. Conversely, Theorem~\ref{volume-to-bir} converts $H^2\ge c\eps^a$ into $M=O(\eps^{-1-a/2})$.
\end{proof}

However, the equivalence here does not preserve the exponent order. In particular, applying the converse only to \eqref{refined-M} gives an eleventh-order volume gap with a logarithmic loss, weaker than the direct ninth-order estimate. No pseudo-effective threshold bound is required for either direction of this surface conversion.

\section{Sharpness and further questions}\label{examples-sec}
The first two families establish the optimal quadratic
anticanonical exponent and the optimal seventh-order
polarisation and cubic threshold exponents. We then
compare them with the Markov family, where the
anticanonical volume stays fixed while the polarisation
volume tends to zero.
\subsection{Sharp quadratic anticanonical order}
\begin{example}\label{quadratic-example}
For an integer $u\ge1$ put $r=u^2+u+1$, and let
$\mu_r$ act on $\PP^2$ by
\[
                 [x:y:z]\longmapsto[x:\zeta y:\zeta^{u+1}z].
\]
Let $Y_u$ be the quotient. The three character differences
$1,u,u+1$ are coprime to $r$, so the quotient map is
quasi-\'etale and its only singularities are the images of
the coordinate points. All three have cyclic type
$\frac1r(1,u+1)$, up to interchanging coordinates:
$-u(u+1)\equiv1\pmod r$ identifies the other two charts.
Thus $Y_u$ is a rank-one klt Fano surface and
\[
             (-K_{Y_u})^2=\frac9r,\qquad
             \eps_u:=\mld(Y_u)=\frac{u+2}{r}.
\]
To verify the second equality, minimize
$j+((u+1)j\bmod r)$ for $1\le j<r$.
For $j\ge u+2$ the sum is larger than $u+2$.
For $1\le j\le u$, no reduction modulo $r$ occurs and
the sum is $(u+2)j$. For $j=u+1$ it is $2u+1\ge u+2$.
The minimum is therefore $u+2$, attained at $j=1$.
This is the cyclic quotient age formula, equivalently the
discrepancy calculation on its toric resolution.
Consequently
\[
 \frac{(-K_{Y_u})^2}{\eps_u^2}
       =\frac{9(u^2+u+1)}{(u+2)^2}\longrightarrow9.
\]
For every $p<2$, $(-K_{Y_u})^2/\eps_u^p\to0$.
This proves sharpness of the exponent in
Theorem~\ref{main-rank-fano}.
\end{example}

\subsection{Sharp seventh-order volume and cubic threshold}
\begin{theorem}\label{seventh-example}
Let $n\ge2$ be even, and put
\[
 a=n^2-1,\quad b=n^2+1,\quad c=2n^3,\qquad
 X_n=\PP(a,b,c),\quad H_n=\mathcal O_{X_n}(1).
\]
Then $X_n$ is a well-formed rank-one Fano surface and
\begin{align}
 \eps_n:=\mld(X_n)&=\frac1n,\label{seventh-mld}\\
 H_n^2&=\frac1{2n^3(n^4-1)},\label{seventh-volume}\\
 \tau(X_n,H_n)=\lambda(X_n,H_n)&=2n^2(n+1),\label{cubic-example}\\
 (-K_{X_n})^2&=\frac{2n(n+1)}{(n-1)(n^2+1)}.\label{seventh-anti}
\end{align}
In particular $H_n^2\sim\frac12\eps_n^7$ and
$\tau(X_n,H_n)\sim2\eps_n^{-3}$.
\end{theorem}
\begin{proof}
The weights are pairwise coprime: $a,b$ are odd with
$\gcd(a,b)=1$, and each is coprime to $2n^3$.
In the lattice $N=\Z^2$ take the primitive fan rays
\[
                u=(0,-1),\qquad v=(b,n),\qquad w=(-a,n).
\]
Their consecutive determinants are $b,c,a$, and
$cu+av+bw=0$. The greatest common divisor of the
determinants is one, so these rays generate $N$ and define
the stated weighted projective plane. We use the usual
toric discrepancy description; see \cite{Fulton}.

Let $P$ be the triangle with vertices $u,v,w$, and let
$\psi$ be its gauge, the piecewise linear function equal
to one on these three rays. A primitive lattice vector
$z$ defines a toric valuation of log discrepancy $\psi(z)$.
For $z=(x,y)\in N\setminus\{0\}$, if $y\ge1$, the maximal
height $n$ of $P$ gives $\psi(z)\ge y/n\ge1/n$.
If $y\le-1$, the minimal height $-1$ gives
$\psi(z)\ge-y\ge1$. Finally
\[
        P\cap\{y=0\}=
        \left[-\frac a{n+1},\frac b{n+1}\right]\times\{0\}.
\]
For a nonzero lattice point on this line, the gauge is at
least $(n+1)/b$ or $(n+1)/a$, both greater than $1/n$.
Conversely $(0,n)$ lies on the upper edge, so
$\psi(0,1)=1/n$.

This calculation also controls non-toric divisorial
valuations. On a smooth toric resolution the exceptional
boundary has simple normal crossings; subsequent point
blowups have discrepancies $2$, $1+\alpha$, or
$\alpha+\beta$. None is smaller than $1/n$, and every
divisorial valuation over a smooth surface is obtained
by point blowups. This proves \eqref{seventh-mld}.
Weighted intersection theory \cite{Dolgachev}
gives $H_n^2=1/(abc)$, while
\[
              -K_{X_n}\sim(a+b+c)H_n=2n^2(n+1)H_n.
\]
These identities prove all the other formulas.
\end{proof}

\begin{corollary}\label{seventh-cy-example}
No uniform estimate $H^2\ge C\eps^p$ with $C>0$ and
$p<7$ holds for all rank-one $\eps$-lc log Calabi--Yau
pairs with integral ample $H$. No uniform estimate
$\tau(X,H)\le C\eps^{-p}$ with $p<3$ holds even for
rank-one $\eps$-lc Fano surfaces.
\end{corollary}
\begin{proof}
Choose a sufficiently divisible $N$ and a general
$D\in|-NK_{X_n}|$ that is smooth and avoids the singular
points. Arrange $1-1/N\ge1/n$. Then $B_n=D/N$ makes
$(X_n,B_n)$ an $\eps_n$-lc log Calabi--Yau pair. By
\eqref{seventh-volume},
$H_n^2/\eps_n^p\sim\frac12\eps_n^{7-p}\to0$ for $p<7$.
By \eqref{cubic-example},
$\tau(X_n,H_n)\eps_n^p\sim2\eps_n^{p-3}\to+\infty$
for $p<3$.
\end{proof}

The same family also shows the sharpness of the common
Weil-index exponent in Proposition~\ref{rank-one-index}:
since $\Cl(X_n)=\Z[H_n]$ and the weights are pairwise
coprime, its common Cartier multiple is
$abc=2n^3(n^4-1)\asymp\eps_n^{-7}$.

\begin{remark}\label{seventh-bir-obstruction}
Proposition~\ref{bir-volume-converse} also gives a necessary
growth rate for birational systems on this family:
\[
 |mH_n|\text{ birational}\quad\Longrightarrow\quad
 m\ge\frac1{\sqrt{H_n^2}}
   =\sqrt{2n^3(n^4-1)}\asymp\eps_n^{-7/2}.
\]
This is a lower obstruction, not a calculation of the
actual birationality threshold. It leaves a gap from
the sufficient rank-one bound $O(\eps^{-9/2})$.
\end{remark}
\subsection{The Markov family: fixed anticanonical volume}
The Markov weighted planes of Hacking--Prokhorov
\cite[Theorem 1.1]{HP} give a useful comparison:
their anticanonical volume is fixed, whereas their
integral Weil polarisation volumes tend to zero.

\begin{example}\label{markov-example}
There are Fano surfaces $X_i$ with integral ample Weil
divisors $H_i$ and actual discrepancy thresholds
$\eps_i\to0$ such that
\[
 H_i^2\asymp\eps_i^4,\qquad
 \tau(X_i,H_i)=\lambda(X_i,H_i)\asymp\eps_i^{-2},
 \qquad (-K_{X_i})^2=9.
\]
Each admits an $\eps_i$-lc log Calabi--Yau $\Q$-boundary.
\end{example}
\begin{proof}
Starting with $(b,c)=(1,2)$, iterate
$(b,c)\mapsto(c,3c-b)$. This preserves
\[
                         1+b^2+c^2=3bc,
\]
as well as coprimality and primeness to three.
The entries tend to infinity and, after the first
mutation, $2<c/b<3$. Put $X=\PP(1,b^2,c^2)$ and
$H=\mathcal O_X(1)$. The nontrivial coordinate charts
are cyclic quotients of the form
$\frac1{n^2}(1,na-1)$ with $\gcd(a,n)=1$, where
$n=b$ or $c$; for example $a\equiv3c\pmod b$ in the
first chart. We use the weighted formulas of
\cite{Dolgachev}.

Such a quotient has minimal log discrepancy $1/n$.
Indeed, by \cite[Section 4]{Reid} its discrepancy is
the minimum of
\[
             \frac{j+\overline{(na-1)j}}{n^2},
                    \qquad 1\le j<n^2,
\]
where the bar is the least positive residue modulo $n^2$.
Every numerator is a positive multiple of $n$.
Choose $1\le j<n$ with $aj\equiv1\pmod n$;
then the residue is $n-j$, giving equality.
Consequently the actual discrepancy threshold is $\eps=1/c$.

Weighted intersection theory gives
\[
       H^2=\frac1{b^2c^2}\asymp\eps^4,
       \qquad -K_X\sim3bcH.
\]
Thus the thresholds are $3bc\asymp\eps^{-2}$ and
$(-K_X)^2=9$. A general anticanonical member with
sufficiently small coefficient gives the required
boundary, as in Proposition~\ref{big-boundary-thm}.
\end{proof}
\subsection{Further questions}
The first question is whether every $\eps$-lc log
Calabi--Yau surface pair with an integral nef and big Weil
polarisation satisfies $H^2\ge c\eps^7$ for an absolute
$c>0$. Theorem~\ref{seventh-example} forces exponent at
least seven, while Theorem~\ref{main} gives exponent nine
with a logarithmic loss. In the big-boundary branch our
proof combines $\vol(-K_X)\gg\eps^3/\ellog(\eps)$
with $\tau(X,H)\ll\eps^{-3}$. A quadratic general
anticanonical bound would give exponent eight by this
comparison alone. Reaching seven would require a further
relation between the anticanonical volume and the threshold,
or a direct intersection argument. The two sharp families
above also show that the separate extremal orders need
not occur on the same surface.

The threshold reduction does not transfer the rank-one
volume bound back to the original surface.
In the notation of Lemma~\ref{scaling-step},
$H=g^*H_Z-uC$ with $u\ge0$ at a divisorial MMP step, so
\[
                     H^2=H_Z^2+u^2C^2\le H_Z^2.
\]
Thus the MMP preserves the threshold but can increase
the polarisation volume.

The second question concerns the factor $\eps^{-1}$ in
Theorem~\ref{volume-to-bir}. Can this loss be reduced?
For rank-one ample models the theorem gives the sufficient
order $\eps^{-9/2}$, whereas
Remark~\ref{seventh-bir-obstruction} gives only a necessary
order $\eps^{-7/2}$. Determining the actual birationality
threshold in the seventh-order family would help clarify
this gap.

More generally, a polynomial volume gap remains open for
projective $\eps$-lc surfaces with an integral nef and big
Weil divisor $H$ satisfying $H-K_X$ pseudo-effective,
without a log Calabi--Yau assumption.
The threshold theorem and the volume-to-birationality
conversion apply in this setting, but our volume argument
uses the discrepancy bound for the pair to control its
anticanonical models. Ambient $\eps$-lc singularities on
a surface of Fano type need not provide an anticanonical
boundary with the same discrepancy bound.

\appendix
\section{Classification tables and numerical estimates}\label{classification-app}
This appendix records the complete ranges used in
Proposition~\ref{classification-reduction}. The external
classification supplies the geometric coverage and
admissible parameter ranges; the continuant formulas
and discrepancy estimates of Section~\ref{rank-one-sec}
give the numerical bounds recorded here. All finite
lists and unbounded parameter families are retained.
\subsection{The published Lacini correspondence}
The published classification \cite{Lacini}, building on \cite{KeelMcKernan}, has 24 series; the older arXiv version has 22. We use the published numbering and the correspondence in \cite[Section 7, Tables 3--4]{PP3}. The notation is fixed in Table~\ref{tab:classification-notation}.
Here $C$ denotes a noncanonical cyclic quotient, $F$ a noncanonical quotient fork, and $A$ a canonical point. The following table accounts for every published label. ``Fixed'' means that no unbounded chain parameter remains in that entry. ``Rays'' refers to the three explicit unbounded rays in $A_{23}$, not to arbitrary chains.

\begingroup\small
\begin{longtable}{@{}p{.10\textwidth}p{.39\textwidth}p{.39\textwidth}@{}}
\caption{Correspondence with the published Lacini series}\label{tab:lacini-crosswalk}\\
\toprule
LDP & Correspondence over $\mathbb C$ & Singularity filter and disposition\\\midrule\endfirsthead
\caption[]{Correspondence with the published Lacini series (continued)}\\
\toprule
LDP & Correspondence over $\mathbb C$ & Singularity filter and disposition\\\midrule\endhead
\bottomrule\endlastfoot

1 & \cite[Lemma 6.6]{PP3} & $2C+A$; fixed.\\
2--4 & The fixed baskets listed in Subsection~\ref{r1-higher} & $2C$; fixed.\\
5 & $B_1$ & $2C$; fixed.\\
6 & A fixed basket in Subsection~\ref{r1-higher} & $2C+A$; fixed.\\
7 & $B_{14}$ & $F+C+A$; two parameter values.\\
8 & $A_{15}$ & $F+C$; fixed.\\
9 & $B_3$ for seed $A_1+A_5$; a fixed remaining basket for seed $3A_2$ & $2C$ or $2C+A$; fixed.\\
10 & $B_8,B_9$ & $2C+A$; fixed.\\
11 & $B_8$ & $2C+A$; fixed, overlapping LDP10.\\
12 & $B_{14}$ & $F+C+A$; overlaps LDP7.\\
13 & (1): $A_2,A_3$; (2): $A_4,A_6$; (3): $B_6$; (4): $B_4,B_5$; (5)--(8): fixed remaining baskets; (9): $B_{10},B_{11}$; (10): $B_{12}$ & Two noncanonical cyclic points, with zero or one canonical companion; all these parameters are bounded.\\
14 & $A_{23}$ & $2C$; retain its allowed rays and finite part.\\
15 & $A_{13}$ & $F+C$; three parameter values.\\
16 & $A_{22}$ & $F+C$; fixed.\\
17 & (1),(2): $A_{23}$; (3): $A_5,A_9$ & $2C$; the first two subseries require the $A_{23}$ calculation.\\
18 & $A_{23}$ & $2C$; same parameter ranges.\\
19 & $A_{11}$ & $F+C$; fixed.\\
20 & Height at most two & Use the complete Tables 9--11 case analysis below.\\
21 & Characteristic-five example & Absent over $\mathbb C$.\\
22 & Height at most two, or an elliptic descendant & Already reduced to the low-height or one-noncanonical-point arguments.\\
23 & Height at most two or elliptic descendants, except the seed branches 5.1(3)--(5) & The exceptional seed branches are precisely the higher-height families listed in Subsection~\ref{r1-higher}.\\
24 & Height at most two & Use Tables 9--11.\\
\end{longtable}
\endgroup

The supplementary baskets $A_{16}$ and $B_2$
in \cite{PP3}, noted there as apparently absent from
Lacini's list, are both included. The possibly omitted surface in \cite[Example 3.13]{Nagaoka} has basket
\[
 [2,3,2,2]+[2,3,2]+[2,2,2],
 \quad V=1/22,\quad \mld(X)=5/11;
\]
it is explicitly included in the finite-basket calculation. The exact assertion that all the remaining higher-height baskets have height four is announced for a subsequent paper in \cite{PP3}. Our estimate only uses their displayed baskets and does not rely on that announced height equality.

Belousov's theorem \cite[Theorems 1.1 and 2.1]{Belousov} bounds the total number of singularities by four. The refined correspondence gives a stronger filter here: every higher-height basket requiring a new check has exactly two noncanonical points, either $2C$ or $F+C$, with at most one canonical companion. Arbitrarily variable configurations with three noncanonical cyclic points occur in the low-height lists, where the section or toric geometry already controls them.

\subsection{Low-height cases}
The notation is as in Table~\ref{tab:classification-notation}. A displayed constant $C$ means $V\ge\eps^2/C$. For product estimates the assertion follows from \eqref{r1-eq:volume}; fixed canonical companions are omitted. ``Excluded'' means non-klt, not merely absent from a numerical search.

\begingroup\small
\begin{longtable}{@{}p{.16\textwidth}p{.12\textwidth}p{.64\textwidth}@{}}
\caption{Low-height volume estimates: $V\ge\eps^2/C$}\label{tab:low-height}\\
\toprule
Rows & $C$ & Justification\\\midrule\endfirsthead
\caption[]{Low-height volume estimates: $V\ge\eps^2/C$ (continued)}\\
\toprule
Rows & $C$ & Justification\\\midrule\endhead
\bottomrule\endlastfoot

$9(1\!:\!3),9(5)$ & $1$ & The height-one section calculation in \eqref{r1-heightone}, $V\ge\eps$.\\
$9(4),9(6)$ & $2$ & One noncanonical cyclic point.\\
$9(7)$ & $15/2$ & The direct section estimate $V\ge\min\{\eps,2/15\}$.\\
$9(8),9(10)$ & $4$ & At most one noncanonical fork.\\
$9(9)$ & $16$ & A fork and a cyclic companion of determinant at most four.\\
$9(11),9(12)$\newline $9(19)$ & $16$ & A fork and the companion $[3,\two{b-3}]$, of determinant $2b-3\le4/\eps$.\\
$9(13)$ & $6$ & $h_F=1$; the variable companion has determinant at most $3b\le6/\eps$.\\
$9(14)$ & $18$ & Arms $(2,3,d)$, $3\le d\le5$: $h_F=6-d$, and the companion determinant is at most $d(b-1)\le2d/\eps$.\\
$9(15\!:\!17)$ & $4$ & At most one noncanonical point.\\
$9(18)$ & $3$ & Formula \eqref{r1-eq:row18}.\\
$9(20),9(21)$ & excluded & Parabolic arm triple $(2,3,6)$.\\
$9(22)$ & $4$ & Only $r=2$ is klt; $h_F=1$, and its extra $[r]$ is canonical.\\
$9(23)$ & $24$ & Only $r=2,3$ are klt; use $h_F=4-r$, the companion bound $6/\eps$, and $r\le3$.\\
$9(24)$ & $16$ & A fork and a fixed cyclic companion of determinant at most four.\\
$9(25\!:\!27)$ & excluded & Non-klt benches.\\\midrule
$10(1),10(2)$ & $1$ & Toric and coupled-cyclic arguments, respectively.\\
$10(3),10(12)$\newline $10(22),10(25)$ & $2$ & One noncanonical cyclic point.\\
$10(4),10(7)$\newline $10(8),10(14)$\newline $10(23)$ & excluded & Parabolic fork arms.\\
$10(5)$ & $40$ & $m=2,3,4$; $h_F=5-m$, fixed determinant $2m-1$, and variable determinant $2r-3\le4/\eps$.\\
$10(6)$ & $4$ & $h_F=1$ and variable companion determinant at most $4/\eps$.\\
$10(9),10(10)$ & $28$ & A fork and a fixed companion of determinant at most seven; row 10 is better.\\
$10(11)$ & $1/2$ & Lemma~\ref{r1-lem:sections}.\\
$10(13),10(16)$ & $32$ & Fork determinant at most $4/\eps$, cyclic determinant at most $8/\eps$.\\
$10(15)$ & $8$ & Proposition~\ref{r1-prop:row15}.\\
$10(17)$ & $7$ & $h_F=1$, fixed cyclic determinant at most seven.\\
$10(18)$ & $26$ & $h_F=7-m-n>0$; the companion has determinant $6m-7$ or $6m-5$, with $m+n\le6$.\\
$10(19)$ & $20$ & $d=4,5$, $h_F=6-d$, companion determinant at most $3d-2$.\\
$10(20)$ & $15$ & Proposition~\ref{r1-prop:row20}.\\
$10(21),10(29)$ & $16$ & Only two forks can be noncanonical.\\
$10(24)$ & $12$ & $m=2,3$, $h_F=3,1$; variable companion at most $4/\eps$.\\
$10(26)$ & $36$ & Fork bound $4/\eps$, companion bound $9/\eps$.\\
$10(27)$ & $18$ & $h_F=1$, companion at most $3(n+4)\le18/\eps$.\\
$10(28)$ & $45$ & $h_F=6-d$, $3\le d\le5$; companion at most $d(n+3)\le5d/\eps$, and $d(6-d)\le9$.\\
$10(30)$ & excluded & A non-klt bench.\\
$5.15(1),(2)$ & $16$ & One noncanonical cyclic point, or two forks.\\\midrule
$11(3)$ & $2$ & Its forks are canonical; only the cyclic point may be noncanonical.\\
$11(4)$ & $32$ & One noncanonical fork and a cyclic companion of determinant $4n\le8/\eps$.\\
$11(5)$ & $16$ & At most two noncanonical forks.\\
$11(6)$ & $32$ & Two noncanonical forks, or a fork and a fixed companion of determinant eight.\\
\end{longtable}
\endgroup

All bounded-twig estimates in this table are elementary continuants. For example $d([3,\two\ell])=2\ell+3$ and $d([3,\two\ell,3])=4\ell+8$. A parabolic arm triple has $c=0$ in \eqref{r1-eq:fork}, so cannot occur on a klt surface. Thus the low-height part is bounded by $45$, by the displayed determinant estimates.

\subsection{Remaining parameter families and finite baskets}\label{r1-higher}
The correspondence above includes all 24 published Lacini series, the additional $A_{16},B_2$ baskets, and the Nagaoka correction.

For a two-cyclic ray, choose the displayed parameter vertex whose discrepancy is $N/D_1$. Since $\eps\le N/D_1$,
\[
                   D_1D_2\eps^2\le N^2D_2/D_1.
\]
The quotient of two positive affine functions has constant derivative sign. Its supremum on $k\ge k_0$ is therefore the larger of its initial value and its limit. The following table proves a uniform constant less than $40$ on every such ray.

\begin{table}[htbp]\centering\small
\caption{Bounds for two-cyclic parameter families}\label{tab:cyclic-rays}
\begin{tabular}{@{}lrrrrr@{}}
\toprule
Family & $D_1$ & $D_2$ & $N$ & $k_0$ & Bound\\\midrule
$A_1$ & $18k-27$ & $2k-1$ & $9$ & $4$ & $63/5$\\
$A_7$ & $12k-16$ & $3k-1$ & $8$ & $3$ & $128/5$\\
$A_8$ & $16k-24$ & $2k+1$ & $8$ & $3$ & $56/3$\\
$B_7$ & $8k-10$ & $4k-1$ & $6$ & $3$ & $198/7$\\
$B_{13}$ & $16k-24$ & $2k-1$ & $8$ & $3$ & $40/3$\\
$H_{4a}$ & $9k-12$ & $3k-1$ & $6$ & $3$ & $96/5$\\
$A_{23}$, ray 1 & $8k-10$ & $4k+3$ & $6$ & $3$ & $270/7$\\
$A_{23}$, ray 2 & $9k-3$ & $3k+8$ & $6$ & $3$ & $51/2$\\
$A_{23}$, ray 3 & $6k-1$ & $6k+5$ & $5$ & $3$ & $575/17$\\\bottomrule
\end{tabular}
\end{table}

Here $H_{4a}$ is the $4A_2$-seed branch of \cite[Section 7, LDP23]{PP3}, with basket $[2,2,k,2,2]+[\two{k-2},3,2]+[2,2]$. The numerator $N$ is the sum of the determinants of the two sides of the chosen vertex; it is not an estimate of an unrelated cyclic singularity.

The second $4A_2$-seed branch $H_{4b}$ has a fork with arms $(2,3,3)$, central weight $k$, cyclic companion $[3,\two{k-3},3,2]$, and canonical companion $[2,2]$. For fork--cyclic rays retain the exact $h_F$ in \eqref{r1-eq:fork}. The following constants bound $D_FD_C\eps^2$; if a specialisation becomes canonical, dropping its determinant only helps.

\begin{table}[htbp]\centering\small
\caption{Bounds for fork--cyclic parameter families}\label{tab:fork-rays}
\begin{tabular}{@{}p{.22\textwidth}p{.51\textwidth}r@{}}
\toprule
Families & Data and estimate & Constant\\\midrule
$A_{10}$ & $h_F=4$, $D_C=6k-5$, $k\eps\le1$ & $24$\\
$A_{12},A_{19}$ & $h_F=2$, $D_C=12k+5$; separate $k=2$ & $58$\\
$A_{14},A_{21}$ & Arms $(2,3,d)$; $D_C\le6dk$, $h_F=6-d$ & $96$\\
$A_{17}$ & $h_F=4$, $D_C=8k+2$, $k\ge3$ & $35$\\
$A_{18}$ & $h_F=2$, $D_C=12k+1$, $k\ge3$ & $25$\\
$A_{20}$ & $h_F=3$, $D_C=6k+7$, $k\ge3$ & $25$\\
$B_{16}$ & Arms $(2,4,d)$, $d=2,3$; $D_C\le4dk$ & $64$\\
$H_{4b}$ & $h_F=3$, $D_C=6k-5$, $k\ge3$ & $18$\\\bottomrule
\end{tabular}
\end{table}

For clarity, in $A_{14},A_{21}$, the $k\ge3$ estimate is $6d(6-d)\le54$. The allowed noncanonical $k=2$ cases have $d=4,5$, giving at most $12d(6-d)\le96$ using $\eps\le1$. In $B_{16}$, the analogous $k\ge3$ bound is $32$, and $64$ safely includes $k=2$ whenever it is admissible.

The remaining checks are finite; we give the input chains and rational
bounds explicitly. For a chain $[b_1,\ldots,b_s]$, let $l_i,r_i$
be the determinants strictly to the left and right of vertex $i$.
Its minimal log discrepancy is
\[
 \mu=\min\left\{1,\min_i\frac{l_i+r_i}{d([b_1,\ldots,b_s])}\right\}.
\]
The minimum over the resolution computes the ambient minimum:
further point blowups have discrepancy $2$, $1+\alpha_i$, or
$\alpha_i+\alpha_j$, and cannot decrease it.
For two cyclic points the table records $D_1D_2\mu^2$, where
$\mu$ is the minimum of their two values. For a fork and a cyclic
point it records $4D_C\mu_C$; indeed
$D_FD_C\eps^2\le4D_C\eps\le4D_C\mu_C$.
The symbol $F+$ suppresses the fork, whose exact type is immaterial
to this bound. Canonical companions are suppressed throughout this
finite table, since they do not enter \eqref{r1-eq:volume}.
All entries follow from \eqref{r1-eq:continuant} and the displayed
minimum, so no external computation is required.
\begingroup\small
\begin{longtable}{@{}p{.17\textwidth}p{.61\textwidth}r@{}}
\caption{Finite baskets and their determinant bounds}\label{tab:finite-baskets}\\
\toprule
Family & Noncanonical cyclic chains (and a fork, if present) & Bound\\
\midrule\endfirsthead
\caption[]{Finite baskets and their determinant bounds (continued)}\\
\toprule
Family & Noncanonical cyclic chains (and a fork, if present) & Bound\\
\midrule\endhead
\bottomrule\endlastfoot

$(\star)$ & $[2,2,3,\two{5}]+[3,2]$ & $15$\\
$A_{2}$ & $[2,2,3,3,\two{5}]+[3,2]$ & $845/73$\\
$A_{9}$ & $[2,3,2,3,\two{3}]+[2,3,2,2]$ & $396/13$\\
$B_{1}$ & $[2,3,\two{5}]+[3,2,2]$ & $112/5$\\
$B_{3}$ & $[2,3,\two{5}]+[2,4,2,2]$ & $500/17$\\
$B_{4}$ & $[2,3,3,2,2]+[2,3,\two{4}]$ & $1088/29$\\
$B_{6}$ & $[2,2,3,\two{6}]+[3]$ & $300/31$\\
$B_{9}$ & $[2,5,\two{4}]+[2,3,\two{3}]$ & $686/37$\\
$B_{10}$ & $[2,3,3,\two{3}]+[2,3,2,2]$ & $891/37$\\
$B_{12}$ & $[\two{3},3,\two{4}]+[3]$ & $243/29$\\
\cite[Lemma 6.6]{PP3} & $[2,3,\two{3}]+[3,\two{3}]$ & $162/7$\\
LDP2--4 & $[2,2,3,2]+[\two{5},3]$ & $325/11$\\
LDP6 & $[2,2,3,2,2]+[3,2,2]$ & $84/5$\\
LDP9, second & $[2,4,2,2]+[2,2,3,2,2]$ & $375/17$\\
LDP13.5 & $[2,2,3,\two{3}]+[3,2]$ & $245/19$\\
LDP13.6 & $[2,3,\two{3}]+[3,\two{4}]$ & $198/7$\\
LDP13.7 & $[2,3,\two{6}]+[3,2]$ & $405/23$\\
LDP13.8a & $[2,3,2,2]+[2,3,2]$ & $200/11$\\
LDP13.8b & $[2,3,\two{4}]+[2,2,3]$ & $343/17$\\
$A_{3}(3)$ & $[\two{3},3,\two{5}]+[3]$ & $150/17$\\
$A_{5}(3)$ & $[2,2,3,3,\two{4}]+[3,2,2]$ & $504/31$\\
$A_{6}(3)$ & $[2,3,3,\two{5}]+[2,3,2]$ & $968/53$\\
$B_{5}(3)$ & $[2,2,3,2,2]+[3,\two{4}]$ & $132/5$\\
$B_{8}(3)$ & $[2,4,\two{3}]+[2,3,\two{3}]$ & $252/11$\\
$B_{11}(3)$ & $[2,2,3,\two{3}]+[3,2,2]$ & $343/19$\\
$A_{3}(4)$ & $[\two{3},4,\two{6}]+[3]$ & $363/67$\\
$A_{5}(4)$ & $[2,2,3,4,\two{4}]+[3,\two{3}]$ & $1296/97$\\
$A_{6}(4)$ & $[2,4,3,\two{5}]+[2,3,2,2]$ & $1859/79$\\
$B_{5}(4)$ & $[2,2,4,\two{3}]+[3,\two{4}]$ & $539/31$\\
$B_{8}(4)$ & $[2,2,4,\two{3}]+[2,4,\two{3}]$ & $1078/31$\\
$B_{11}(4)$ & $[2,2,4,\two{4}]+[3,2,2]$ & $224/19$\\
$A_{4}(4)$ & $[2,2,4,\two{6}]+[3,2]$ & $125/13$\\
$A_{4}(5)$ & $[2,2,5,\two{7}]+[3,2]$ & $605/83$\\
$A_{11}$ & $F+[3,3,\two{3}]$ & $28$\\
$A_{15}$ & $F+[2,3,2,2]$ & $20$\\
$A_{16}$ & $F+[3,2,4,2,2]$ & $32$\\
$A_{22}$ & $F+[3,4,\two{5}]$ & $36$\\
$B_{2}$ & $F+[3,\two{4}]$ & $24$\\
$B_{15}$ & $F+[2,4,2,2]$ & $20$\\
$A_{13}(4)$ & $F+[2,4,\two{4}]$ & $28$\\
$A_{13}(5)$ & $F+[2,5,\two{5}]$ & $32$\\
$A_{13}(6)$ & $F+[2,6,\two{6}]$ & $36$\\
$B_{14}(3)$ & $F+[2,3,\two{3}]$ & $24$\\
$B_{14}(4)$ & $F+[2,4,\two{3}]$ & $24$\\

\end{longtable}\endgroup
Here $(\star)$ is the exceptional basket
$[2,2,3,\two5]+[3,2]$ in the comparison preceding
\cite[Table 4]{PP3}; LDP labels refer to the published
classification. The first LDP13(8) basket, with its canonical
$[2,2,2]$ companion restored, also covers the potentially omitted
surface of \cite[Example 3.13]{Nagaoka}. The maximum in this
44-entry table is $1088/29<38$.

It remains to spell out the finite part of $A_{23}$.
Its basket, from \cite[Lemma 4.12(23), Table 2]{PP3}, is
\begin{equation}\label{r1-A23}
 [\two{c-1},b,d,a,\two{b-1}]+[\two{a-1},c+1,\two d].
\end{equation}
The unbounded rays are $(k,2,2,2)$, $(3,2,k,2)$ and
$(2,3,k,2)$, $k\ge3$, already treated above. The remaining
14 Cartesian products, comprising 38 tuples, are as follows;
$[u,v]$ in this table means the set of integers from $u$ to $v$.

\begin{table}[htbp]\centering\small
\caption{The finite parameter ranges in $A_{23}$}\label{tab:A23}
\begin{tabular}{@{}rrrrr@{}}
\toprule
$a$&$b$&$c$&$d$&$\max D_1D_2\mu^2$\\\midrule
$[6,8]$&2&3&2&$2205/47$\\
$[3,7]$&2&2&3&$1944/43$\\
5&2&$[3,4]$&2&$2816/53$\\
4&2&$[3,7]$&2&$1323/31$\\
4&2&2&4&$725/16$\\
$[3,4]$&3&2&2&$1815/41$\\
3&2&2&$[4,7]$&$368/11$\\
3&2&4&3&$2205/43$\\
2&$[4,7]$&2&2&$22$\\
2&5&3&2&$1700/73$\\
2&4&$[3,5]$&2&$3509/71$\\
2&4&3&3&$4312/95$\\
2&3&3&3&$2662/51$\\
2&3&2&$[2,6]$&$539/17$\\\bottomrule
\end{tabular}
\end{table}

Applying the continuant and discrepancy formulas to the two chains
in \eqref{r1-A23} gives
\[
 \max D_1D_2\mu^2=\frac{2816}{53}<54,
 \quad\text{at }(a,b,c,d)=(5,2,4,2).
\]
For example this tuple has determinants $53,44$ and $\mu=8/53$,
giving $53\cdot44\cdot(8/53)^2=2816/53$. The finite parameter
ranges above are the full ranges in the classification, not
truncations of infinite families.

\section{Finite arithmetic for local indices}\label{arithmetic-app}
We record the finite calculations used in the four-point
case and in the numerically trivial case.

\begin{lemma}\label{four-orders-arithmetic}
There are exactly $126$ integer quadruples
$2\le g_1\le g_2\le g_3\le g_4$ satisfying
\[
 \sum_{i=1}^4\frac1{g_i}\ge1,
 \qquad \sum_{i=1}^3\frac1{g_i}<1.
\]
Their least common multiples are at most $1722$.
\end{lemma}
\begin{proof}
The first inequality gives $g_1\le4$ and
$g_2\le\lfloor3/(1-1/g_1)\rfloor$. Put
$r=1-1/g_1-1/g_2$. In this branch $r>0$, and the two
inequalities give exactly the finite ranges
\[
 \max\{g_2,\lfloor1/r\rfloor+1\}
       \le g_3\le\lfloor2/r\rfloor,
 \qquad
 g_3\le g_4\le\left\lfloor\frac1{r-1/g_3}\right\rfloor.
\]
Enumerating these ranges gives Table~\ref{tab:four-orders}.
The largest least common multiple is attained at
$(2,3,7,41)$.
\begin{table}[htbp]
\centering\small
\caption{Four local group orders with reciprocal sum at least one.}
\label{tab:four-orders}
\begin{tabular}{@{}lrrr@{}}
\toprule
$g_1$ & $2$ & $3$ & $4$\\\midrule
Number of quadruples & $109$ & $16$ & $1$\\
Maximum of $\lcm(g_1,g_2,g_3,g_4)$ & $1722$ & $132$ & $4$\\
\bottomrule
\end{tabular}
\end{table}
\end{proof}

\begin{lemma}\label{ade-budget}
For a collection of ADE singularities of total resolution
rank at most nineteen, the least common multiple $q$ of
the local class-group exponents is at most $840$.
\end{lemma}
\begin{proof}
The exponents for $A_r,D_r,E_6,E_7,E_8$ are respectively
$r+1$, either $2$ or $4$, $3$, $2$, and $1$.
For an $A_r$ singularity,
\[
           \sum_{p^a\Vert r+1}(p^a-1)\le r,
\]
by repeated use of $u+v-2\le uv-1$ for $u,v\ge2$.
The analogous inequality follows directly from the
displayed exponents for the other types. Assign each
maximal prime power in $q$ to one component whose exponent
contains it. Summing the resulting rank costs gives
\begin{equation}\label{ADE-budget-eq}
                  \sum_{p^a\Vert q}(p^a-1)\le19.
\end{equation}
Only primes at most nineteen can occur. For each such
prime, choose either no factor or one of its powers
$p^a\le20$, retain choices whose total cost in
\eqref{ADE-budget-eq} is at most nineteen, and multiply
the chosen powers. Table~\ref{tab:ade-budget} gives the
maxima, grouped by largest prime. The case $q=1$ is
immediate. The maximum $840=8\cdot3\cdot5\cdot7$ has
cost $7+2+4+6=19$.
\begin{table}[htbp]
\centering\small
\caption{Products of prime powers under the ADE rank budget.}
\label{tab:ade-budget}
\begin{tabular}{@{}lrrrrrrrr@{}}
\toprule
Largest prime & $2$ & $3$ & $5$ & $7$ & $11$ & $13$ & $17$ & $19$\\\midrule
Maximum product & $16$ & $72$ & $360$ & $840$ & $660$ & $390$ & $102$ & $38$\\
\bottomrule
\end{tabular}
\end{table}
\end{proof}

\end{document}